\documentclass[12pt]{amsart}
\usepackage{aliascnt}
\usepackage{amssymb}
\usepackage{amsmath}
\usepackage[margin=1in]{geometry}
\usepackage{enumitem}
\usepackage{xcolor}
\definecolor{darkblue}{rgb}{0.0, 0.0, 0.8}
\usepackage[colorlinks=true, allcolors=darkblue]{hyperref}
\usepackage{mathtools}

\newtheorem{theorem}{Theorem}[section]

\newaliascnt{lemma}{theorem}
\newtheorem{lemma}[lemma]{Lemma}
\aliascntresetthe{lemma}

\newaliascnt{proposition}{theorem}
\newtheorem{proposition}[proposition]{Proposition}
\aliascntresetthe{proposition}

\newaliascnt{corollary}{theorem}
\newtheorem{corollary}[corollary]{Corollary}
\aliascntresetthe{corollary}

\newaliascnt{conjecture}{theorem}

\aliascntresetthe{conjecture}

\newaliascnt{assumption}{theorem}

\aliascntresetthe{assumption}

\newaliascnt{definition}{theorem}
\newtheorem{definition}[definition]{Definition}
\aliascntresetthe{definition}

\newaliascnt{question}{theorem}

\aliascntresetthe{question}

\newaliascnt{remark}{theorem}
\newtheorem{remark}[remark]{Remark}
\aliascntresetthe{remark}

\newaliascnt{example}{theorem}

\aliascntresetthe{example}

\newtheorem*{notation*}{Notation}
\newtheorem*{theorem*}{Theorem}
\newtheorem*{conjecture*}{Conjecture}

\numberwithin{figure}{section}
\numberwithin{table}{section}
\numberwithin{equation}{section}

\allowdisplaybreaks

\newcommand{\bN}{\mathbb N}

\newcommand{\bR}{\mathbb R}

\newcommand{\cH}{\mathcal H}

\newcommand{\cK}{\mathcal K}

\newcommand{\cM}{\mathcal M}

\newcommand{\cU}{\mathcal U}
\newcommand{\cV}{\mathcal V}

\newcommand{\dist}{\mathrm{dist}}

\newcommand{\id}{\mathrm{id}}

\newcommand{\ip}[2]{\left\langle #1,#2\right\rangle_{\cH}}
\newcommand{\norm}[1]{\left\lVert #1\right\rVert_{\cH}}
\newcommand{\opnorm}[1]{\left\lVert #1\right\rVert_{\mathrm{op}}}

\usepackage[capitalise,nameinlink]{cleveref}
\crefname{theorem}{theorem}{theorems}
\Crefname{theorem}{Theorem}{Theorems}
\crefname{lemma}{lemma}{lemmas}
\Crefname{lemma}{Lemma}{Lemmas}
\crefname{proposition}{proposition}{propositions}
\Crefname{proposition}{Proposition}{Propositions}
\crefname{corollary}{corollary}{corollaries}
\Crefname{corollary}{Corollary}{Corollaries}
\crefname{conjecture}{conjecture}{conjectures}
\Crefname{conjecture}{Conjecture}{Conjectures}
\crefname{assumption}{assumption}{assumptions}
\Crefname{assumption}{Assumption}{Assumptions}
\crefname{definition}{definition}{definitions}
\Crefname{definition}{Definition}{Definitions}
\crefname{question}{question}{questions}
\Crefname{question}{Question}{Questions}
\crefname{remark}{remark}{remarks}
\Crefname{remark}{Remark}{Remarks}
\crefname{example}{example}{examples}
\Crefname{example}{Example}{Examples}
\usepackage{autonum}
\title[Conformal Flexibility of Completeness]
{On Conformal Flexibility of Completeness\\ and Minimizing Geodesics
on Hilbert Manifolds}

\author{Levin Maier}
\address{Faculty of Mathematics and Computer Science,
	University of Heidelberg}
\email{lmaier@mathi.uni-heidelberg.de}


\begin{document}
	\begin{abstract}
This article is part of a broader programme investigating which features of
finite-dimensional Riemannian geometry persist in infinite dimensions. The
Hopf--Rinow theorem fails even for Hilbert manifolds: metric and geodesic
completeness need not agree, and neither property guarantees a
length-minimizing geodesic between prescribed endpoints.

Despite this failure, our first main result shows that the conformal class of
every smooth strong Riemannian metric on a smooth separable Hilbert manifold
contains a smooth strong representative that is metrically and geodesically
complete and such that every two points in the same connected component are
joined by a length-minimizing geodesic.

By contrast, our second main result establishes a local flexibility phenomenon
for metric completeness that is inherently infinite-dimensional. Given any
prescribed Hilbert-norm ball, a metrically complete strong metric admits a
conformal deformation which is equal to one outside that ball, preserves
geodesic completeness, and destroys metric completeness.
\end{abstract}
	\maketitle
	\tableofcontents
	\section{Introduction}
	\subsection*{Context}

Infinite-dimensional manifolds arise naturally when the configurations of a
system are themselves functions, maps, or transformations.  Examples include
fluid configurations in hydrodynamics, probability densities in optimal
transport, spaces of shapes and immersions in geometric data analysis, and
diffeomorphisms in computational anatomy
\cite{Arnold66,BenamouBrenier2000,Otto2001,Younes1998,
BauerBruverisMichor2014Overview,SrivastavaKlassen2016,
BegMillerTrouveYounes2005}. Riemannian geometry provides a natural language for these spaces: a metric
measures the cost of an infinitesimal deformation, geodesics describe optimal
motions, and completeness asks whether such motions can break down globally. These applications highlight the need for a systematic understanding of
geometric structures on infinite-dimensional manifolds.  

This need has motivated a broader program, of which the present article forms
part, aimed at understanding which features of classical finite-dimensional Riemannian geometry persist in infinite dimensions. Over the past several decades, beginning with the pioneering work of Eells, Elworthy, Omori, and others (see, for example, \cite{Palais68,Eells1966,Eliasson1967,La99,Omori1974,EellsElworthy1970} and the references therein), significant parts of Riemannian geometry have been extended to infinite-dimensional settings, including the theory of connections, geodesic equations, curvature, and variational methods.

At the same time, it has become clear that infinite-dimensional geometry exhibits phenomena with no finite-dimensional counterpart.  Perhaps the best-known example is the failure of the Hopf--Rinow theorem,
first observed by Atkin~\cite{HopfrinowfalseAktkin} and later demonstrated by Ekeland through a different construction~\cite{HopfrinowfalseEkeland}.
Unlike in finite dimensions, geodesic completeness and metric completeness are not equivalent, and neither property guarantees the existence of a minimizing geodesic between every pair of points in the same connected
component.  At the variational level, the basic obstruction is the loss of compactness: a minimizing sequence of curves may converge only weakly, while the energy need not be lower semicontinuous under this weak convergence.

Another genuinely infinite-dimensional phenomenon, which has motivated substantial research over the past three decades, is the possible vanishing of geodesic distance for weak Riemannian and weak Finsler metrics: the induced distance may vanish identically even though the infinitesimal metric is pointwise nondegenerate; see \cite{EliashbergPolterovich1993,Michor_Mumford_vanishing_geodesic_distance, BauerHarmsPreston2020,BauerBruverisHarmsMichor2012,JerrardMaor2019}.

Against this background, it is natural to revisit another foundational finite-dimensional result, the Nomizu--Ozeki theorem, which asserts that every finite-dimensional Riemannian manifold is conformally equivalent to a complete Riemannian manifold~\cite{NomizuOzeki1961}.  In infinite dimensions, Biliotti--Mercuri proved that every separable Hilbert manifold equipped with a strong Riemannian metric is conformally equivalent to a metrically complete
Riemannian metric~\cite[Theorem~2.1]{BiliottiMercuri2017}.  For strong metrics, metric completeness also implies geodesic completeness~\cite{La99}.  However, Atkin showed that this still does not guarantee the existence of a minimizing geodesic between every two points, even on an ellipsoid in Hilbert space~\cite{HopfrinowfalseAktkin}.

The remaining endpoint-minimization assertion is not a marginal addition to metric completeness.  It is the part of Hopf--Rinow that turns local Riemannian geometry into a globally solvable two-point variational problem. Metric completeness rules out finite-length escape, but it does not force a minimizing sequence of curves to converge to a minimizer.  Thus the central
question is whether conformal freedom can overcome this core global failure of the Hopf--Rinow theorem on Hilbert manifolds.

\subsection*{Main results} 
Before stating our two main results, we recall the class of metrics considered
throughout.  A Riemannian metric \(G\) on a separable Hilbert manifold \(\cM\)
is called \emph{strong} if, for every \(x\in\cM\), the norm induced by \(G_x\)
yields the given topology on \(T_x\cM\); see
\Cref{Def:Strong and weak metrics} for further details. With this terminology, our first main result shows that this obstruction can always be overcome by a conformal change: every smooth strong Riemannian metric admits a conformal representative with the three principal global conclusions of Hopf--Rinow:	
\begin{theorem}[=\Cref{thm:main-theorem}]\label{thm:intro-full-hopf-rinow}
Let \(\cM\) be a separable Hilbert manifold, and let \(G\) be a smooth strong Riemannian metric on \(\cM\). Then the conformal class of \(G\) contains a smooth strong Riemannian metric that is metrically complete and geodesically complete, and such that every two points in the same connected component of \(\cM\) are joined by a length-minimizing geodesic.
\end{theorem}
The proof does not merely add a boundary barrier to the known conformal completion construction.  Its new feature is an inverse-metric/Fenchel-duality argument: inverse-metric control is converted into weak lower semicontinuity of the energy, thereby making minimizing geodesics survive weak limits. \Cref{thm:intro-full-hopf-rinow} shows that conformal deformation can create a complete geometry with globally minimizing geodesics.  This raises a complementary question: how rigid is metric completeness under further
conformal changes?  In infinite dimensions, it is not rigid at all.  Our second main result realizes this failure locally, inside an arbitrarily prescribed Hilbert-norm ball, while preserving geodesic completeness.
	\begin{theorem}[=\Cref{thm:local-conformal-atkin}]
		\label{thm:intro-local-atkin}
		Let \(\cH\) be an infinite-dimensional separable real Hilbert space, let \(\cU\subset\cH\) be a nonempty open subset, and let \(G\) be a metrically complete smooth strong Riemannian metric on \(\cU\). For every Hilbert-norm ball \(W\subset\cH\) with \(\overline W\subset\cU\), where the closure is taken in \(\cH\), there exists a smooth function
		\[
		\rho\colon\cU\longrightarrow(0,1],
		\qquad
		\rho=1\quad\text{on }\cU\setminus W,
		\]
		such that \(\rho G\) is geodesically complete but not metrically complete.
	\end{theorem}
The force of \Cref{thm:intro-local-atkin} lies in its relative and local character.  No finite-dimensional analogue is possible.  Indeed, if \(W\) had relatively compact closure and \(\rho>0\) were smooth with \(\rho=1\) outside \(W\), then \(\rho\) would be bounded above and below by positive constants.  Hence \(G\) and \(\rho G\) would be uniformly equivalent, and metric completeness of \(G\) could not be destroyed by such a deformation.

In an infinite-dimensional Hilbert space, by contrast, a bounded ball is noncompact.  The theorem exploits this noncompactness to create a finite-length escaping end inside an arbitrarily prescribed ball, while retaining geodesic completeness.  Thus it is not merely another example separating metric and geodesic completeness: it is a relative realization of the Atkin phenomenon inside any chosen local region of an arbitrary complete background geometry.

Taken together, \Cref{thm:intro-full-hopf-rinow} and \Cref{thm:intro-local-atkin} reveal a striking conformal flexibility on Hilbert manifolds.  Starting from any smooth strong Riemannian metric, \Cref{thm:intro-full-hopf-rinow} produces a conformal representative that is metrically complete, geodesically complete, and admits minimizing geodesics between arbitrary endpoints.  Applying \Cref{thm:intro-local-atkin} to this representative then yields a further conformal deformation, localized in an arbitrarily prescribed Hilbert-norm ball, that remains geodesically complete but is no longer metrically complete.  Thus every such conformal class contains both types of metric.  This global nonclosedness phenomenon is formulated precisely in \Cref{thm:atkin-deformation}.
\subsection*{New mechanisms and proof ideas}
The two main theorems rest on two distinct infinite-dimensional mechanisms. For \Cref{thm:intro-full-hopf-rinow}, the central issue is to recover endpoint minimizers despite the loss of compactness.  For \Cref{thm:intro-local-atkin}, the objective is to separate metric and geodesic completeness by a deformation which is both relative and local.

\noindent\emph{Barrier--duality for endpoint minimizers.}
The novelty behind \Cref{thm:intro-full-hopf-rinow} is not merely a conformal
barrier preventing escape towards the boundary.  Such barriers can yield
metric completion, but they do not provide the compactness needed to realize
the distance between prescribed endpoints: minimizing sequences generally
converge only weakly, while their energy need not be weakly lower
semicontinuous.  Our construction simultaneously builds a barrier and controls the inverse quadratic forms with the weak upper-semicontinuity needed for the Fenchel-duality argument.
Fenchel duality converts this inverse-metric information into weak sequential
lower semicontinuity of the extended energy.  The completeness criterion of
Bauer, Maor, and Wirth~\cite{bauer2025completeness} then yields metric
completeness, geodesic completeness, and minimizing geodesics.  Thus the
essential new step is a variational completion mechanism, rather than metric
completion alone.  The construction is first carried out on an open subset of
a separable Hilbert space, obtained through the Eells--Elworthy
open-embedding theorem, and is then pulled back to the original manifold.

\noindent\emph{Localized Atkin ends and dynamical escape control.}
Atkin constructed a globally defined strong Riemannian metric for which metric and geodesic completeness separate~\cite{Atkin97}.  The new contribution in \Cref{thm:intro-local-atkin} is to upgrade this global model to a relative local surgery on an arbitrary complete background metric.  Given any prescribed Hilbert-norm ball, we construct a conformal factor which is identically equal to one outside that ball and implant an Atkin end inside it. 
The construction has to preserve two opposed features under this localization. A finite-length escaping curve through the implanted end creates metric incompleteness, while an independent virial/momentum estimate excludes escape of a geodesic in finite affine time.  Thus the result is not merely another Atkin-type example: it realizes the separation conformally, relative to an arbitrary complete metric, and inside an arbitrarily prescribed local region. This local conformal-flexibility mechanism relies essentially on the noncompactness of bounded balls in infinite-dimensional Hilbert spaces.
\subsection*{Further directions}
The two main theorems provide starting points for several existing and emerging
research programs in infinite-dimensional geometry.

\paragraph{\emph{Periodic geodesics beyond half-Lie groups.}}
The systematic study of nonconstant periodic geodesics in infinite dimensions
has recently begun for half-Lie groups~\cite{BauerMaier2026}.  It is natural
to ask whether the complete global-minimizer geometries supplied by
\Cref{thm:intro-full-hopf-rinow} can serve as a background for variational
constructions of periodic geodesics on more general Hilbert manifolds.  The
theorem does not itself produce closed geodesics, but it supplies one useful
global attainment ingredient for such a theory; further topological and
compactness input would still be required.

\paragraph{\emph{Equivariant completion and shape spaces.}}
A concrete application question is whether the conformal completion in
\Cref{thm:intro-full-hopf-rinow} can be made compatible with the symmetries of
Sobolev manifolds of immersions.  If a reparametrization-invariant construction
can be achieved on a setting with a well-behaved quotient, it could descend to
shape space and provide complete geometries with globally minimizing matching
paths.  This would connect the theorem directly to elastic shape analysis and
diffeomorphic matching
\cite{Younes1998,BauerBruverisMichor2014Overview,SrivastavaKlassen2016,
BegMillerTrouveYounes2005}.

\paragraph{\emph{Beyond quadratic geodesic energy.}}
The barrier--duality mechanism suggests an extension from quadratic
Riemannian energy to fiberwise uniformly convex, superlinear Tonelli
Lagrangians.  One may ask for barrier-type or Finsler-type hypotheses ensuring
endpoint action minimizers together with complete Euler--Lagrange dynamics.
Such a theory would connect the present construction to weak KAM and
Aubry--Mather questions in infinite dimensions
\cite{TonelliHalfLiegroups,HopfRinowHalfLiegroups}.

\paragraph{\emph{Relative and parametric conformal flexibility.}}
\Cref{thm:intro-local-atkin} gives a relative local implantation of one Atkin
end.  Can one implant locally finite collections of ends, prescribe
finite-length escaping curves, or construct smooth parameter families of
conformal deformations while retaining geodesic completeness?  A positive
answer would turn the local theorem into the beginning of a genuine relative
and parametric conformal-flexibility theory.
\subsection*{Overview of the article}
	
\Cref{s:preliminaries} recalls the notions of strong Riemannian metrics, metric and geodesic completeness, and the completeness criterion used in the proof. It also records the Eells--Elworthy open-embedding theorem, which reduces the global problem to an open subset of a separable Hilbert space.

\Cref{sec:conformal Hopf-Rinow} proves the conformal full Hopf--Rinow theorem. We first construct a controlled conformal barrier metric and then verify the weak lower-semicontinuity of its extended energy by means of inverse-metric estimates and Fenchel duality. The abstract completeness criterion then yields metric completeness, geodesic completeness, and the existence of minimizing geodesics.
	
\Cref{s:local-conformal-flexibility} establishes the localized Atkin deformation and derives the global conformal-flexibility result. It shows that metric completeness can be destroyed inside an arbitrarily prescribed Hilbert-norm ball while geodesic completeness is preserved.
	
The following technical subsections contain the technical ingredients for the local construction: the virial escape criterion excluding finite-time geodesic escape, and the construction of the localized Atkin end.
	
\subsection*{Acknowledgements}The author thanks M. Bauer, C. Maor and P.Michor for helpful comments and suggestions.
The author acknowledges funding from the Deutsche
Forschungsgemeinschaft (DFG, German Research Foundation) through grants
281869850 (RTG 2229), 390900948 (EXC-2181/1), and 281071066 (TRR 191).

The author used ChatGPT during the preparation of this article for proofreading, language refinement, and mathematical exploration. In particular, these discussions helped refine the presentation of the barrier--duality argument and the localized conformal construction. The author takes full responsibility for all mathematical statements, proofs, and conclusions in this article.

	\section{Preliminaries: Riemannian Metrics on Infinite-Dimensional Manifolds}\label{s:preliminaries}
	
	In this section, we introduce the basic notation and definitions from infinite-dimensional Riemannian geometry that will be used throughout the article. For a more comprehensive introduction to the subject, we refer the reader to~\cite{La99,schmeding2022introduction,maor2022riemannian}. In particular, we recall a recent completeness result for strong Riemannian metrics in infinite-dimensional settings~\cite{Bauer_2025,bauer2025completeness}.
	
	We start by recalling the notion of a Riemannian metric in the infinite-dimensional setting. In contrast to the finite-dimensional setting, one has to distinguish between two different types of Riemannian metrics, usually referred to as weak and strong Riemannian metrics.
	
	\begin{definition}[Strong and weak metrics]\label{Def:Strong and weak metrics}
		Let $\mathcal M$ be a smooth infinite-dimensional manifold. A Riemannian metric $G$ on $\mathcal M$ is a smooth map
		\[
		G \colon T\mathcal M \times_{\mathcal M} T\mathcal M \to \bR
		\]
		such that $G_x$ is an inner product on $T_x\mathcal M$ for every $x \in \mathcal M$. The metric $G$ is called a \emph{strong Riemannian metric} if the inner product $G_x$ induces the manifold topology on $T_x\mathcal M$ for every $x \in \mathcal M$. It is called a \emph{weak Riemannian metric} if there exists $x\in \mathcal M$ such that the topology induced by $G_x$ on $T_x\mathcal M$ is weaker than the natural manifold topology.
	\end{definition}
	
	Since the existence of a strong Riemannian metric necessarily implies that the underlying manifold is of Hilbert type, cf.~\cite{Bauer_2025}, and since we are mostly interested in strong Riemannian manifolds in this article, we restrict our attention from now on to Hilbert manifolds. The following statement is a specialization of a recent result of Bauer, Maor, and Wirth. It provides sufficient conditions for the three global properties used in this
article: metric completeness, geodesic completeness, and the existence of
minimizing geodesics between arbitrary endpoints.
	
	\begin{theorem}[Completeness for strong Riemannian metrics~{\cite[Theorem 2.1]{bauer2025completeness}}]\label{thm:abstract_completeness}
		Let $\mathcal M$ be an open subset of a Hilbert space $({\mathcal H},\langle \cdot,\cdot \rangle_{{\mathcal H}})$, and let $G$ be a smooth strong Riemannian metric on $\mathcal M$. Assume that the following conditions hold:
		\begin{enumerate}[label=(\alph*)]
			\item \label{item:abc-a} On every $G$-metric ball $B$ in $(\mathcal M,G)$, there exists $C>0$ such that
			$\langle h,h\rangle_{\mathcal H} \leq C G_x(h,h)$
			for all $x\in B$ and all $h\in T_x\mathcal M$.
			
			\item \label{item:abc-b} The closure of $B$ in ${\mathcal H}$, taken with respect to the Hilbert space topology, lies in $\mathcal M$.
			
			\item \label{item:abc-c} For every $p,q\in\mathcal M$, the extended energy
			\[
			\mathcal E_G\colon H^1_{p,q}([0,1],\mathcal H)\longrightarrow [0,\infty],
			\qquad
			\mathcal E_G(\gamma):=
			\begin{cases}
				\displaystyle\int_0^1G_{\gamma(t)}(\dot\gamma(t),\dot\gamma(t))\,dt,
				&\gamma([0,1])\subset\mathcal M,\\[0.6em]
				+\infty,&\gamma([0,1])\not\subset\mathcal M,
			\end{cases}
			\]
			is sequentially weakly lower semicontinuous, where
			\[
			H^1_{p,q}([0,1],\mathcal H)
			:=\{\gamma\in H^1([0,1],\mathcal H):\gamma(0)=p,\ \gamma(1)=q\}.
			\]
		\end{enumerate}
		Then $(\mathcal M,G)$ is metrically and geodesically complete, and there exists a $G$-minimizing geodesic between any two points in the same connected component of $(\mathcal M,G)$.
	\end{theorem}
	
	We note that, in the infinite-dimensional setting considered in this article, the Hopf--Rinow theorem is famously false in general~\cite{Atkin97,HopfrinowfalseEkeland}. Thus, the existence of minimizing geodesics is not an automatic consequence of metric completeness and has to be proved separately. Indeed, condition~\ref{item:abc-c} is not needed for metric completeness, but only for the existence of minimizing geodesics. This part of the argument is based on the direct method in the calculus of variations, for which weak lower semicontinuity is a standard ingredient; see~\cite[Section~2]{bauer2025completeness}.
	
	We close this section by recalling the global topological input. It is well known that every separable
	Hilbert manifold is parallelizable~\cite{BurgheleaKuiper1969}. Hence the
	open-embedding theorem of Eells--Elworthy~\cite[Theorem~1A]{EellsElworthy1970}
	yields the following result.
	
	\begin{theorem}[Eells--Elworthy]
		\label{thm:all_HM_are-open-subsets-of-HS}
		Let $\cH$ be an infinite-dimensional separable real Hilbert space, and let
		$\cM$ be a smooth separable metrizable manifold modeled on $\cH$. Then there
		are an open subset $\cU\subset\cH$ and a diffeomorphism
		\(\Phi\colon\cM\to\cU\).
	\end{theorem}
	\section{Conformal Completion and Endpoint Minimization}\label{sec:conformal Hopf-Rinow}
	The following theorem is the main theoretical advance of this article. It
	shows that, on a separable Hilbert manifold, the conformal class of every
	smooth strong Riemannian metric contains a metric satisfying the full
	Hopf--Rinow conclusion.
	
	\begin{theorem}[=\Cref{thm:intro-full-hopf-rinow}]\label{thm:main-theorem}
		Let \(\cM\) be a separable Hilbert manifold, and let \(G\) be a smooth strong
		Riemannian metric on \(\cM\). Then the conformal class of \(G\) contains a
		smooth strong Riemannian metric that is metrically complete, geodesically
		complete, and such that every two points in the same connected component of
		\(\cM\) are joined by a length-minimizing geodesic.
	\end{theorem}
	
	We reduce this global statement to a local assertion on an open subset of a
	Hilbert space.
	
	\begin{theorem}\label{thm:local-conformal}
		Let \(\cH\) be an infinite-dimensional separable real Hilbert space, let
		\(\cU\subset\cH\) be open, and let \(G\) be a smooth strong Riemannian
		metric on \(\cU\). Then there exists a smooth function \( \rho\colon\cU\longrightarrow(0,\infty)\) such that the conformal metric \(\widetilde G:=\rho G\) has the following properties:
		\begin{enumerate}[label=\textup{(\roman*)}]
			\item \label{it:1_local-conformal} \((\cU,d_{\widetilde G})\) is metrically complete;
			\item \label{it:2_local-conformal} \((\cU,\widetilde G)\) is geodesically complete;
			\item \label{it:3_local-conformal} every two points in the same connected component of \(\cU\) are
			joined by a length-minimizing \(\widetilde G\)-geodesic.
		\end{enumerate}
	\end{theorem}
	
	Before proving \Cref{thm:local-conformal}, we show that it implies the global
	result.
	
	\begin{proof}[Proof of \Cref{thm:main-theorem}]
		Let \(G\) be a smooth strong Riemannian metric on \(\cM\). By
		\Cref{thm:all_HM_are-open-subsets-of-HS}, there exist an infinite-dimensional
		separable real Hilbert space \(\cH\), an open subset \(\cU\subset\cH\), and a
		diffeomorphism \(\Phi\colon\cM\longrightarrow\cU\).
		
		Push \(G\) forward to \(\cU\) via \(\Phi\), that is, define \(  G_{\cU}:=(\Phi^{-1})^*G.\) This is a smooth strong Riemannian metric on \(\cU\). Applying
		\Cref{thm:local-conformal} to \((\cU,G_{\cU})\), we obtain a smooth function
		\(\rho_{\cU}\colon\cU\to(0,\infty)\) such that
		\((\cU,\rho_{\cU}G_{\cU})\) is metrically complete, geodesically complete,
		and any two points in the same connected component are joined by a
		length-minimizing geodesic.
		
		Define \( \rho:=\rho_{\cU}\circ\Phi\), then 
		\[
		\Phi^*\left(\rho_{\cU}G_{\cU}\right)
		=(\rho_{\cU}\circ\Phi)G
		=\rho G.
		\]
		Thus \(\Phi\) is an isometry between \((\cM,\rho G)\) and
		\((\cU,\rho_{\cU}G_{\cU})\). Since metric completeness, geodesic completeness,
		and the existence of length-minimizing geodesics are preserved by isometries,
		\((\cM,\rho G)\) has all the properties asserted in
		\Cref{thm:main-theorem}.
	\end{proof}
The rest of the section is devoted to the proof of \Cref{thm:local-conformal}.
	\subsection{Proof of \texorpdfstring{\Cref{thm:local-conformal}}{Theorem 3.2}}
	\label{sec:local-conformal-proof}
	For the proof of \Cref{thm:local-conformal}, it suffices to construct, for
	every smooth strong Riemannian metric $G$ on an open subset
	\(\cU\subseteq\cH\), a smooth function $\rho$ such that
	\((\cU,\rho G)\) satisfies the hypotheses of the abstract completeness
	theorem~\Cref{thm:abstract_completeness}.
	
	The proof is divided into three steps. We first construct a conformal
	metric with the required barrier estimates. We then verify the hypotheses of
	\Cref{thm:abstract_completeness}. Finally, we apply that theorem to conclude
	\Cref{thm:local-conformal}.
	\subsubsection{Construction of the conformal metric}
	\label{sec:construction}
	
	We construct a conformal change of \(G\) with two features. It will create a
	barrier that makes the boundary of \(\cU\) increasingly expensive to approach,
	and it will provide quantitative control of the inverse metric for the later
	weak-convergence argument.
	
	First, recall that every smooth strong Riemannian metric \(G\) on \(\cU\) is
represented by an operator-norm smooth field of bounded operators, called
its \emph{inertia operator},
	\[
	A\colon\cU\longrightarrow\mathcal L(\cH),
	\]
	such that
	\[
	G_x(u,v)=\ip{A(x)u}{v}
	\qquad\bigl((x,u,v)\in\cU\times\cH\times\cH\bigr).
	\]
	We shall modify this inertia operator by a smoothly varying positive scalar
	factor. The factor will be chosen in a controlled way: the resulting metric
	will grow near the boundary of \(\cU\), while its inverse will retain the
	weak-continuity properties needed later in the energy argument. The following
	proposition makes this controlled conformal deformation precise.
	\begin{proposition}[Controlled conformal deformation]
		\label{prop:conformal-metric}
		There exist smooth maps
		\[
		F\colon\cU\longrightarrow[0,\infty),
		\qquad
		D,B\colon\cU\longrightarrow\mathcal L(\cH),
		\qquad
		\rho\colon\cU\longrightarrow(0,\infty),
		\]
		with the following properties:
		\begin{enumerate}[label=\textup{(\roman*)}]
			\item \label{it:1-controlled-construction}
			For every \(x\in\cU\), the operators \(D(x)\) and \(B(x)\) are bounded,
			positive, self-adjoint, and boundedly invertible. Moreover,
			\[
			0<D(x)\leq\id,
			\qquad
			B(x)=e^{F(x)}D(x)^{-1}.
			\]
			
			\item \label{it:2-controlled-construction}
			The bilinear form
			\[
			\widetilde G_x(u,v):=\ip{B(x)u}{v}
			=\rho(x)G_x(u,v),
			\qquad (x,u,v)\in\cU\times\cH\times\cH,
			\]
			defines a smooth strong Riemannian metric on \(\cU\).
			
			\item \label{it:3-controlled-construction}
			The metric coefficient satisfies
			\[
			B(x)\geq\id
			\qquad(x\in\cU).
			\]
			
			\item \label{it:4-controlled-construction}
			If \(\cU\neq\cH\), then, with
			\[
			d(x):=\dist_{\cH}(x,\cH\setminus\cU),
			\]
			one has
			\[
			B(x)\geq d(x)^{-2}\id
			\qquad(x\in\cU).
			\]
		\end{enumerate}
	\end{proposition}
	
	\begin{remark}[Geometric consequence]
		By \ref{it:4-controlled-construction} in \Cref{prop:conformal-metric} one has that \(\widetilde G\) is a conformal barrier metric: it dominates the ambient
		Hilbert norm and blows up towards the boundary of \(\cU\).
	\end{remark}
	
	The proof of \Cref{prop:conformal-metric} uses two ingredients. The first
	provides smooth majorants and a barrier against weak convergence to the
	complement of \(\cU\). We begin with the first ingredient: 
	\begin{lemma}[Smooth weakly lower semicontinuous majorant]
		\label{lem:majorant}
		Let $b\colon \cU\to[1,\infty)$ be a locally bounded function. Then there exists a
		function
		\[
		\overline c\colon \cH\longrightarrow[1,\infty]
		\]
		with the following properties:
		\begin{enumerate}[label=\textup{(\roman*)}]
			\item \label{it:1-majorant}
			$c:=\overline c|_{\cU}$ is smooth and $c(x)\geq b(x)$ for every
			$x\in \cU$;
			\item \label{it:2-majorant}
			$\overline c(x)<\infty$ exactly when $x\in \cU$;
			\item \label{it:3-majorant}
			$\overline c$ is sequentially weakly lower semicontinuous on $\cH$.
		\end{enumerate}
	\end{lemma}
	
	\begin{proof}
		Since \(\cH\) is separable, the open set \(\cU\) is second countable. We
		may therefore choose a countable family of Hilbert balls
		\[
		B(c_j,r_j),
		\qquad j\in\bN,
		\]
		which covers \(\cU\), together with radii \(0<r_j<R_j\) such that
		\[
		\overline{B(c_j,R_j)}\subset\cU
		\]
		and \(b\) is bounded on \(B(c_j,R_j)\). Next, choose integers
		\(N_j\in\bN\) satisfying
		\[
		N_j\geq j,
		\qquad
		N_j\geq\sup_{B(c_j,R_j)}b.
		\]
		For every \(j\in\bN\), choose a smooth nondecreasing function
		\[
		\theta_j\colon[0,\infty)\longrightarrow[0,1]
		\]
		such that
		\[
		\theta_j(t)=0
		\quad\text{for }t\leq r_j^2,
		\qquad
		\theta_j(t)=1
		\quad\text{for }t\geq R_j^2.
		\]
		For \(n\in\bN\), define the index set
		\[
		J_n:=\{j\in\bN:N_j\leq n\}.
		\]
		This set is finite. Indeed, if \(j\in J_n\), then \(N_j\leq n\), and
		hence \(j\leq n\) because \(N_j\geq j\). Thus
		\(J_n\subset\{1,\ldots,n\}\). Next, for every \(n\in\bN\), define
		\begin{equation}\label{eq:eta-n}
			\eta_n\colon\cH\longrightarrow[0,1],
			\qquad
			\eta_n(x):=
			\prod_{j\in J_n}\theta_j\bigl(\norm{x-c_j}^2\bigr),
		\end{equation}
		where the empty product is interpreted as \(1\). Each \(\eta_n\) is smooth
		on \(\cH\). Indeed, the map
		\(x\mapsto\norm{x-c_j}^2\) is smooth, and \(J_n\) is finite.
		
		The family \((\eta_n)_{n\in\bN}\) is locally finite on \(\cU\). To see
		this, fix \(x\in\cU\) and choose \(j_0\in\bN\) such that
		\(x\in B(c_{j_0},r_{j_0})\). There is a neighbourhood \(V\) of \(x\) such
		that \(V\subset B(c_{j_0},r_{j_0})\). Hence
		\[
		\theta_{j_0}\bigl(\norm{y-c_{j_0}}^2\bigr)=0
		\qquad(y\in V).
		\]
		If \(n\geq N_{j_0}\), then \(j_0\in J_n\), and therefore
		\[
		\eta_n|_V=0.
		\]
		Thus only finitely many \(\eta_n\) are nonzero on \(V\). We now define
		\begin{equation}\label{eq:cbar}
			\overline c\colon\cH\longrightarrow[1,\infty],
			\qquad
			\overline c(x):=
			1+\sum_{n=1}^{\infty}\eta_n(x),
		\end{equation}
		where the sum is allowed to take the value \(+\infty\). It remains to check that for \(\bar c\) \ref{it:1-majorant}-\ref{it:3-majorant} holds.\\
		\textbf{\Cref{it:1-majorant}:} Fix $x\in \cU$ and choose $j_0$ with
		$x\in B(c_{j_0},r_{j_0})$.  On a sufficiently small neighborhood $V$ of
		$x$, the factor indexed by $j_0$ vanishes.  For every $n\geq N_{j_0}$ one
		has $j_0\in J_n$, hence
		\[
		\eta_n|_V=0.
		\]
		Thus the sum in \eqref{eq:cbar} is locally finite on $\cU$, and consequently
		$c=\overline c|_{\cU}$ is smooth.
		
		To complete the proof of \Cref{it:1-majorant}, fix $x\in \cU$ and let
		$n\in\bN$ satisfy
		$n<b(x)$.  If $j\in J_n$ and $x\in B(c_j,R_j)$, then
		\[
		b(x)\leq N_j\leq n,
		\]
		contrary to $n<b(x)$.  Thus $x\notin B(c_j,R_j)$ for every $j\in J_n$, and
		therefore $\eta_n(x)=1$.  Consequently
		\[
		c(x)\geq 1+\#\{n\in\bN:n<b(x)\}\geq b(x) \quad \forall x\in \cU.
		\]
		\textbf{\Cref{it:2-majorant}.}
		By \Cref{it:1-majorant}, $\overline c$ is finite on $\cU$. Conversely, if
		$x\notin \cU$, then $x\notin B(c_j,R_j)$ for every $j$, because each such
		ball is contained in $\cU$. Hence every factor in \eqref{eq:eta-n} equals
		$1$, so $\eta_n(x)=1$ for every $n$ and therefore by~\eqref{eq:cbar} it holds
		$\overline c(x)=+\infty$ for each $x\in\cH\setminus\cU$.\\
		\textbf{\Cref{it:3-majorant}.} Let $x_k\rightharpoonup x$ in $\cH$. For every $j$, weak lower
		semicontinuity of the norm and monotonicity of $\theta_j$ give
		\[
		\liminf_{k\to\infty}
		\theta_j\bigl(\norm{x_k-c_j}^2\bigr)
		\geq
		\theta_j\bigl(\norm{x-c_j}^2\bigr).
		\]
		Finite products of these nonnegative \([0,1]\)-valued sequentially weakly lower semicontinuous functions are again sequentially weakly lower
semicontinuous, and the same holds for finite sums. Hence every partial sum in \eqref{eq:cbar} has this
		property. Since \eqref{eq:cbar} is the supremum of its partial sums,
		$\overline c$ is sequentially weakly lower semicontinuous.
	\end{proof}
	
	The second ingredient is a locally uniform Taylor estimate. It will allow us
	to choose the conformal factor so as to compensate for the local variation of
	the inverse inertia operator.
	\begin{lemma}[Taylor control]
		\label{lem:taylor}
		Let $D\in C^2(\cU, \mathcal{L}(\cH))$.  There exist functions
		\[
		\delta\colon \cU\longrightarrow(0,1],
		\qquad
		L\colon \cU\longrightarrow[1,\infty)
		\]
		such that $\delta^{-1}$ and $L$ are locally bounded and, whenever
		$\norm{h}<\delta(x)$,
		\begin{equation}\label{eq:taylor-control}
			x+h\in \cU,
			\qquad
			\opnorm{D(x+h)-D(x)-D'(x)[h]}
			\leq L(x)\norm{h}^2.
		\end{equation}
	\end{lemma}
	\begin{proof}
		For \(x\in\cU\), call a radius \(r\in(0,1]\) \emph{admissible at \(x\)} if
		\begin{equation}\label{eq:admissible-radius}
			B(x,2r)\subset\cU,
			\qquad
			\sup_{y\in B(x,2r)}\opnorm{D''(y)}\leq r^{-1}.
		\end{equation}
		Let \(\sigma(x)\) be the supremum of the radii admissible at \(x\). The
		set of admissible radii is nonempty by openness of \(\cU\) and continuity
		of \(D''\) at \(x\). Hence \(\sigma(x)>0\).
		
		If \(r\) is admissible at \(x\) and \(\norm{y-x}<r/2\), then \(r/2\) is
		admissible at \(y\). Indeed,
		\[
		B(y,r)\subset B(x,3r/2)\subset B(x,2r),
		\]
		and
		\[
		\sup_{z\in B(y,r)}\opnorm{D''(z)}
		\leq r^{-1}
		\leq (r/2)^{-1}.
		\]
		Consequently, if \(r\) is admissible at \(x\), then \(\sigma(y)\geq r/2\) whenever \(\norm{y-x}<r/2\). Thus \(\sigma\) is locally bounded away from zero.  Next we define
		\[
		\delta(x):=\frac{\sigma(x)}{4},
		\qquad
		L(x):=\delta(x)^{-1}.
		\]
		Then \(\delta^{-1}\) and \(L\) are locally bounded.
		
		Fix \(h\in\cH\) with \(\norm{h}<\delta(x)\). Choose a radius \(r\)
		admissible at \(x\) such that
		\[
		r>\frac{\sigma(x)}{2}=2\delta(x).
		\]
		Then \(x+th\in B(x,r)\subset\cU\) for every \(t\in[0,1]\). Taylor's formula
		with integral remainder gives
		\[
		D(x+h)-D(x)-D'(x)[h]
		=
		\int_0^1(1-t)D''(x+th)[h,h]\,dt.
		\]
		Since \(x+th\in B(x,2r)\), \eqref{eq:admissible-radius} yields
		\[
		\opnorm{D(x+h)-D(x)-D'(x)[h]}
		\leq \frac{1}{2r}\norm{h}^2
		\leq L(x)\norm{h}^2.
		\]
		This proves \eqref{eq:taylor-control}.
	\end{proof}
	We now combine the two ingredients \Cref{lem:majorant} and \Cref{lem:taylor} to prove
	\Cref{prop:conformal-metric}.
	
	\begin{proof}[Proof of \Cref{prop:conformal-metric}]
		We first make the common construction. Write the prescribed metric as \(G_x(u,v)=\ip{A(x)u}{v}\)
		where \(  A\colon\cU\longrightarrow\mathcal{L}(\cH)\) is smooth in the operator norm, and every \(A(x)\) is bounded, positive,
		self-adjoint, and boundedly invertible. Apply \Cref{lem:majorant} to the
		locally bounded function
		\[
		x\longmapsto \max\{1,\opnorm{A(x)^{-1}}\}.
		\]
		We obtain a smooth function \(s\colon\cU\to[1,\infty)\) satisfying
		\begin{equation}\label{eq:s-normalization}
			s(x)\geq \opnorm{A(x)^{-1}}.
		\end{equation}
		Define
		\begin{equation}\label{eq:C-D}
			C(x):=s(x)A(x),
			\qquad
			D(x):=C(x)^{-1},
		\end{equation}
		and let \( K(x):=\opnorm{C(x)}.\) The normalization in \eqref{eq:s-normalization} gives \(C(x)\geq\id\),
		and hence \(K(x)\geq1\).
		
		We now apply \Cref{lem:taylor} to \(D\), obtaining \(\delta\) and \(L\).
		If \(\cU\neq\cH\), put
		\[
		d(x):=\dist_{\cH}(x,\cH\setminus \cU),
		\qquad x\in \cH.
		\]
		On \(\cU\), define the locally bounded function
		\begin{equation}\label{eq:b-final}
			b(x):=
			\begin{cases}
				\displaystyle
				\max\left\{
				1,
				L(x)K(x),
				\frac{\log K(x)}{\delta(x)^2},
				d(x)^{-2}
				\right\},&\cU\neq\cH,\\[1.2em]
				\displaystyle
				\max\left\{
				1,
				L(x)K(x),
				\frac{\log K(x)}{\delta(x)^2}
				\right\},&\cU=\cH.
			\end{cases}
		\end{equation}
We apply \Cref{lem:majorant} again, this time to \eqref{eq:b-final} and denote the resulting majorant by
		\[
		\overline c\colon \cH\to[1,\infty],
		\qquad
		c:=\overline c|_{\cU}
		\]
		In particular we have the following bounds,
		\begin{equation} \label{eq:c-boundary}
			c(x)\geq L(x)K(x),\quad c(x)\delta(x)^2\geq \log K(x),\quad c(x)\geq d(x)^{-2}\quad(\cU\neq \cH).
		\end{equation}
		We use this to define
		\begin{equation}\label{eq:F-B}
			F:\cU\longrightarrow [0,\infty): x\mapsto c(x)\bigl(1+\norm{x}^2\bigr),
			\qquad
			B:\cU \longrightarrow \mathcal{L}(\cH):x\mapsto e^{F(x)}C(x),
		\end{equation}
		where \(C\) is as in \eqref{eq:C-D}. Together with \eqref{eq:s-normalization}, this yields the conformal factor
		\[
		\rho:\cU\longrightarrow(0,\infty):x\mapsto s(x)e^{F(x)}.
		\]
		This gives rise to the conformally equivalent metric:
		\begin{equation}\label{eq:gtilde}
			\widetilde G_x(u,v):=\ip{B(x)u}{v} \quad \forall (x,u,v)\in \cU\times\cH\times\cH.
		\end{equation}
		It remains to check \ref{it:1-controlled-construction}-\ref{it:4-controlled-construction}.
		
		\smallskip
		\noindent\textbf{\Cref{it:1-controlled-construction}:}
		The operator \(C(x)\) is bounded, positive, self-adjoint, and boundedly
		invertible. Since \(C(x)\geq\id\), its inverse \(D(x)\) has the same
		properties and
		\begin{equation}\label{eq:C-D-bounds}
			C(x)\geq\id,
			\qquad
			0<D(x)\leq\id.
		\end{equation}
		The operator \(B(x)=e^{F(x)}C(x)\) is likewise bounded, positive,
		self-adjoint, and boundedly invertible. Moreover, \(D(x)^{-1}=C(x)\), so
		\begin{equation}\label{eq:B-Dinv_eF}
			B(x)=e^{F(x)}D(x)^{-1}.
		\end{equation}
		All these operator fields are smooth. For later use, positivity and
		self-adjointness of \(C(x)\) also give
		\begin{equation}\label{eq:D-lower}
			D(x)=C(x)^{-1}\geq K(x)^{-1}\id.
		\end{equation}
		
		\smallskip
		\noindent\textbf{\Cref{it:2-controlled-construction}:}
		The functions \(s\) and \(F\) are smooth, and \(s>0\); hence \(\rho\) is
		smooth and positive. By \eqref{eq:F-B},
		\[
		B(x)=e^{F(x)}s(x)A(x)=\rho(x)A(x).
		\]
		Consequently,
		\[
		\widetilde G_x(u,v)
		=\ip{B(x)u}{v}
		=\rho(x)G_x(u,v).
		\]
		Since \(B\) is a smooth field of bounded, positive, self-adjoint, and
		boundedly invertible operators, \(\widetilde G\) is a smooth strong
		Riemannian metric.
		
		\smallskip
		\noindent\textbf{\Cref{it:3-controlled-construction}:}
		Since \(F\geq0\) and \(C\geq\id\),
		\begin{equation}\label{eq:B-ge-I}
			B(x)=e^{F(x)}C(x)\geq\id
			\qquad(x\in\cU).
		\end{equation}
		
		\smallskip
		\noindent\textbf{\Cref{it:4-controlled-construction}:}
		Assume \(\cU\neq\cH\). By \eqref{eq:c-boundary},
		\[
		F(x)=c(x)\bigl(1+\norm{x}^2\bigr)
		\geq c(x)\geq d(x)^{-2}.
		\]
		Since \(e^{F(x)}\geq F(x)\) and \(C(x)\geq\id\),
		\begin{equation}\label{eq:B-boundary}
			B(x)=e^{F(x)}C(x)\geq d(x)^{-2}\id
			\qquad(x\in\cU).
		\end{equation}
	\end{proof}
	
	For the remainder of the proof, fix the conformal factor \(\rho\) and the
	construction data \(F,D,B\) supplied by
	\Cref{prop:conformal-metric}.
	
	\subsubsection{Verification of the completeness criterion}
	\label{sec:completion}
	
	Next we have to verify that the metric defined in \Cref{prop:conformal-metric} satisfies the hypotheses of
	\Cref{thm:abstract_completeness}.
	
	\begin{proposition}[Verification of the completeness criterion]
		\label{prop:completeness-criterion}
		The strong Riemannian metric \(\widetilde G\) on $\cU$ constructed in
		\ref{it:2-controlled-construction} in \Cref{prop:conformal-metric} satisfies \ref{item:abc-a}-\ref{item:abc-c} of
		\Cref{thm:abstract_completeness}.
	\end{proposition}
	
	The proof uses weak upper semicontinuity of the inverse metric and a dual
	characterization of the extended energy. The following lemma is the central point of the
	proof.
	
	\begin{lemma}[Weak upper semicontinuity of inverse quadratic forms]
		\label{lem:inverse-usc} Let \(F,D,B,\) as in \Cref{prop:conformal-metric}. For each $z\in \cH$ we define
		\begin{equation}\label{eq:Qz}
			Q_z(x):=
			\begin{cases}
				e^{-F(x)}\ip{D(x)z}{z},&x\in \cU,\\
				0,&x\notin \cU.
			\end{cases}
		\end{equation}
		Then $Q_z\colon \cH\to[0,\infty)$ is sequentially weakly upper
		semicontinuous.  Moreover,
		\begin{equation}\label{eq:Q-bound}
			0\leq Q_z(x)\leq \norm{z}^2
			\qquad(x,z\in \cH).
		\end{equation}
		For $x\in \cU$ one has
		\[
		Q_z(x)=\ip{B(x)^{-1}z}{z}.
		\]
	\end{lemma}
	\begin{proof}
		For \(x\in\cU\), \eqref{eq:B-Dinv_eF} gives
		\(B(x)^{-1}=e^{-F(x)}D(x)\). Hence \eqref{eq:Qz} yields
		\[
		Q_z(x)=\ip{B(x)^{-1}z}{z}
		\qquad \forall x\in\cU.
		\]
		Moreover, \eqref{eq:C-D-bounds} and \(F\geq0\) give
		\[
		0\leq e^{-F(x)}\ip{D(x)z}{z}\leq\norm{z}^2
		\qquad \forall x\in\cU.
		\]
		Together with \eqref{eq:Qz}, this proves \eqref{eq:Q-bound}.
		
		Let \(x_n\rightharpoonup x\) in \(\cH\), and set
		\begin{equation}\label{eq:ell}
			\ell:=\limsup_{n\to\infty}Q_z(x_n).
		\end{equation}
		By \eqref{eq:Q-bound}, \(0\leq\ell<\infty\). After passing to a
		subsequence, we may assume that \(Q_z(x_n)\to\ell\). If \(\ell=0\), then
		\(\ell\leq Q_z(x)\) is immediate. Thus assume that \(\ell>0\).
		
		For all sufficiently large \(n\), one has \(Q_z(x_n)>\ell/2\). Since
		\(Q_z\) vanishes on \(\cH\setminus\cU\), we may discard finitely many terms
		and suppose that \(x_n\in\cU\) for every \(n\).
		
		Suppose first that \(x\notin\cU\). By \ref{it:2-majorant} of
		\Cref{lem:majorant}, \(\overline c(x)=+\infty\). Since \(\overline c\) is
		sequentially weakly lower semicontinuous by \ref{it:3-majorant} of
		\Cref{lem:majorant}, and since
		\(\overline c(x_n)=c(x_n)\) on \(\cU\) by \ref{it:1-majorant} of
		\Cref{lem:majorant}, we obtain
		\[
		+\infty
		=\overline c(x)
		\leq\liminf_{n\to\infty}\overline c(x_n)
		=\liminf_{n\to\infty}c(x_n).
		\]
		Thus \(c(x_n)\to+\infty\). By \eqref{eq:Qz},
		\eqref{eq:C-D-bounds}, and \eqref{eq:F-B},
		\[
		0\leq Q_z(x_n)
		\leq e^{-F(x_n)}\norm{z}^2
		\leq e^{-c(x_n)}\norm{z}^2
		\longrightarrow0.
		\]
		This contradicts \(\ell>0\). Hence the desired inequality holds whenever
		\(x\notin\cU\).
		
		We may therefore assume that \(x\in\cU\). Put \(h_n:=x_n-x\). Then
		\(h_n\rightharpoonup0\), so \((h_n)\) is bounded. Consequently,
		\((\norm{h_n}^2)\) is a bounded sequence of real numbers and has a
		convergent subsequence. Passing to this subsequence does not affect the
		convergence \(Q_z(x_n)\to\ell\). We may therefore assume that
		\[
		\norm{h_n}^2\longrightarrow r^2
		\]
		for some \(r\geq0\). Since \(h_n\rightharpoonup0\),
		\[
		\norm{x_n}^2
		=\norm{x}^2+\norm{h_n}^2+2\ip{x}{h_n}
		\longrightarrow\norm{x}^2+r^2.
		\]
		Moreover, using \ref{it:3-majorant} of \Cref{lem:majorant} again gives
		\[
		\liminf_{n\to\infty}c(x_n)\geq c(x).
		\]
		Together with \eqref{eq:F-B}, this yields
		\[
		\liminf_{n\to\infty}F(x_n)
		=
		\liminf_{n\to\infty}
		c(x_n)\bigl(1+\norm{x_n}^2\bigr)
		\geq c(x)\bigl(1+\norm{x}^2+r^2\bigr)
		=F(x)+c(x)r^2.
		\]
		This is equivalent to
		\begin{equation}\label{eq:F-gap}
			\liminf_{n\to\infty}\bigl(F(x_n)-F(x)\bigr)
			\geq c(x)r^2.
		\end{equation}
		Next, \eqref{eq:F-gap} implies that, for every \(\varepsilon>0\) and all
		sufficiently large \(n\),
		\[
		Q_z(x_n)
		\leq e^{-F(x)}e^{-c(x)r^2+\varepsilon}
		\ip{D(x_n)z}{z}.
		\]
		Taking the limit superior in this inequality, using \eqref{eq:ell}, and
		then letting \(\varepsilon\downarrow0\), we obtain
		\begin{equation}\label{eq:Q-limsup}
			\ell
			\leq e^{-F(x)}e^{-c(x)r^2}
			\limsup_{n\to\infty}\ip{D(x_n)z}{z}.
		\end{equation}
		
		Put \(u:=\ip{D(x)z}{z}\). By \eqref{eq:D-lower},
		\begin{equation}\label{eq:z-by-u}
			\norm{z}^2\leq K(x)u.
		\end{equation}
		
		It remains to estimate the last factor in \eqref{eq:Q-limsup}. The Taylor
		estimate \eqref{eq:taylor-control} applies precisely when
		\(\norm{h_n}<\delta(x)\). We therefore compare the limiting size \(r\) of
		the increments with the Taylor radius \(\delta(x)\).
		
		\smallskip
		\noindent\textbf{Case I: \(r<\delta(x)\).}
		Then \(\norm{h_n}<\delta(x)\) for all sufficiently large \(n\). We define
		the remainder term by
		\begin{equation}\label{eq:Restterm}
			R_n:=D(x_n)-D(x)-D'(x)[h_n].
		\end{equation}
		By \eqref{eq:taylor-control}, the remainder in \eqref{eq:Restterm}
		satisfies
		\[
		\opnorm{R_n}\leq L(x)\norm{h_n}^2.
		\]
		The map \(h\longmapsto\ip{D'(x)[h]z}{z}\) is a bounded linear functional on
		\(\cH\). Since \(h_n\rightharpoonup0\), weak continuity of bounded linear
		functionals gives
		\[
		\ip{D'(x)[h_n]z}{z}\longrightarrow0.
		\]
		It follows, together with \eqref{eq:z-by-u}, that
		\[
		\limsup_{n\to\infty}\ip{D(x_n)z}{z}
		\leq u+L(x)r^2\norm{z}^2
		\leq \bigl(1+L(x)K(x)r^2\bigr)u.
		\]
		Combining this with \eqref{eq:Q-limsup}, the estimate
		\(c(x)\geq L(x)K(x)\) from \eqref{eq:c-boundary}, and \(1+t\leq e^t\), we obtain
		\[
		\begin{aligned}
			\ell
			&\leq e^{-F(x)}e^{-c(x)r^2}
			\bigl(1+L(x)K(x)r^2\bigr)u \\
			&\leq e^{-F(x)}e^{-c(x)r^2}
			\bigl(1+c(x)r^2\bigr)u \\
			&\leq e^{-F(x)}u
			=Q_z(x).
		\end{aligned}
		\]
		
		\smallskip
		\noindent\textbf{Case II: \(r\geq\delta(x)\).}
		In this case, Taylor control need not apply. Instead, the estimate
		\eqref{eq:C-D-bounds} gives
		\[
		\ip{D(x_n)z}{z}\leq\norm{z}^2.
		\]
		Using \eqref{eq:Q-limsup}, \eqref{eq:z-by-u}, and the estimate
		\(c(x)\delta(x)^2\geq\log K(x)\) from \eqref{eq:c-boundary}, we find
		\[
		\ell\leq e^{-F(x)}e^{-c(x)r^2}\norm{z}^2
		\leq e^{-F(x)}K(x)e^{-c(x)\delta(x)^2}u
		\leq e^{-F(x)}u
		=Q_z(x).
		\]
		Thus \(Q_z\) is sequentially weakly upper semicontinuous.
	\end{proof}
    The construction has two independent roles.  The boundary estimate
\eqref{eq:B-boundary} makes approaching \(\partial\cU\) increasingly
expensive.  Independently, the normalization
\eqref{eq:s-normalization}--\eqref{eq:C-D}, the weakly lower
semicontinuous majorant supplied by \Cref{lem:majorant}, and the exponent
in \eqref{eq:F-B} are designed to control the inverse metric under weak
convergence.  The factor \(1+\norm{x}^{2}\) produces precisely the
norm-defect gain in \eqref{eq:F-gap}; combined with the Taylor estimate of
\Cref{lem:taylor}, it yields the weak upper semicontinuity of the inverse
quadratic forms established in \Cref{lem:inverse-usc}.
	We now turn to the lower-semicontinuity argument. The idea is to express the
	quadratic energy density as a supremum of functionals that are affine in the
	velocity. After integration, these dual functionals will be weakly lower
	semicontinuous.
	
	The underlying pointwise identity is Fenchel duality. Let \(T\) be a bounded,
	positive, self-adjoint, and boundedly invertible operator on \(\cH\). Then
	\begin{equation}\label{eq:fenchel-pointwise}
		\ip{Tv}{v}
		=\sup_{z\in\cH}
		\left(2\ip{z}{v}-\ip{T^{-1}z}{z}\right).
	\end{equation}
	Indeed, for every \(z\in\cH\),
	\[
	2\ip{z}{v}-\ip{T^{-1}z}{z}-\ip{Tv}{v}
	=
	-\ip{T^{-1}(z-Tv)}{z-Tv}.
	\]
	Thus the right-hand side of \eqref{eq:fenchel-pointwise} is bounded above by
	\(\ip{Tv}{v}\), and equality is attained for \(z=Tv\).
	
	Applied pointwise with \(T=B(\gamma(t))\), this identity motivates the
	following dual functionals. Fix arbitrary \(p,q\in\cU\), allowing
	\(p=q\). For
	\(\zeta\in L^2([0,1],\cH)\) and
	\(\gamma\in H^1_{p,q}([0,1],\cH)\), we define
	\begin{equation}\label{eq:Phi}
		\Phi_\zeta(\gamma)
		:=2\int_0^1\ip{\zeta(t)}{\dot\gamma(t)}\,dt
		-\int_0^1Q_{\zeta(t)}(\gamma(t))\,dt,
	\end{equation}
	with $Q_z$ as in \eqref{eq:Qz}. This map
	\[
	(x,z)\longmapsto Q_z(x)
	\]
	is Borel measurable in the norm topology: it is continuous on
	\(\cU\times\cH\) and vanishes on its Borel complement. Consequently,
	\(t\mapsto Q_{\zeta(t)}(\gamma(t))\) is measurable. Moreover,
	\eqref{eq:Q-bound} gives
	\[
	0\leq Q_{\zeta(t)}(\gamma(t))\leq\norm{\zeta(t)}^2
	\qquad \text{for almost every }t\in[0,1].
	\]
	Thus the second integral in \eqref{eq:Phi} is well defined and finite. Next we prove that the functional $\Phi_\zeta$ is lower semicontinuous:
	\begin{lemma}[Lower semicontinuity of the dual functionals]
		\label{lem:Phi-lsc}
		For every \(\zeta\in L^2([0,1],\cH)\), the functional
		\(\Phi_\zeta\) defined in \eqref{eq:Phi} is sequentially weakly lower
		semicontinuous on \(H^1_{p,q}([0,1],\cH)\).
	\end{lemma}
	
	\begin{proof}
		Suppose $\gamma_n\rightharpoonup\gamma$ weakly in $H^1([0,1],\cH)$.  Then
		\[
		\dot\gamma_n\rightharpoonup\dot\gamma
		\quad\text{weakly in }L^2([0,1],\cH).
		\]
		Moreover, evaluation at each $t\in[0,1]$ is a bounded linear map from
		$H^1([0,1],\cH)$ to $\cH$.  Hence
		\[
		\gamma_n(t)\rightharpoonup\gamma(t)
		\quad\text{in }\cH
		\]
		for every $t$.  \Cref{lem:inverse-usc} gives, for almost every $t$,
		\[
		\limsup_{n\to\infty}
		Q_{\zeta(t)}(\gamma_n(t))
		\leq Q_{\zeta(t)}(\gamma(t)).
		\]
		The bound
		\[
		0\leq Q_{\zeta(t)}(\gamma_n(t))\leq\norm{\zeta(t)}^2
		\]
		permits the reverse Fatou lemma, yielding
		\[
		\limsup_{n\to\infty}
		\int_0^1Q_{\zeta(t)}(\gamma_n(t))\,dt
		\leq
		\int_0^1Q_{\zeta(t)}(\gamma(t))\,dt.
		\]
		The first term in \eqref{eq:Phi} converges by weak convergence of the
		velocities.  Therefore
		\[
		\Phi_\zeta(\gamma)
		\leq\liminf_{n\to\infty}\Phi_\zeta(\gamma_n).
		\]
	\end{proof}
	
	The purpose of the dual functionals is now transparent. Direct weak lower
	semicontinuity of the energy is difficult because the metric coefficient
	\(B(\gamma(t))\) depends on the curve. Fenchel duality replaces this
	coefficient by its inverse, whose quadratic forms were arranged to have the
	weak upper-semicontinuity property of \Cref{lem:inverse-usc}.
	
	Indeed, if \(\gamma([0,1])\subset\cU\), then
	\eqref{eq:fenchel-pointwise} and \Cref{lem:inverse-usc} give, for every
	\(t\in[0,1]\),
	\[
	\widetilde G_{\gamma(t)}\bigl(\dot\gamma(t),\dot\gamma(t)\bigr)
	=
	\sup_{z\in\cH}
	\left(
	2\ip{z}{\dot\gamma(t)}
	-Q_z\bigl(\gamma(t)\bigr)
	\right).
	\]
	Thus \(\Phi_\zeta\) is obtained by choosing a time-dependent dual vector
	\(z=\zeta(t)\) in this pointwise formula. We take the supremum over all such
	choices. This is useful because every \(\Phi_\zeta\) is sequentially weakly
	lower semicontinuous by \Cref{lem:Phi-lsc}, and the following lemma will show
	that the resulting supremum is exactly the extended Riemannian energy.
	
	Define
	\begin{equation}\label{eq:E-star}
		\mathcal E^*(\gamma)
		:=\sup_{\zeta\in L^2([0,1],\cH)}\Phi_\zeta(\gamma).
	\end{equation}
	By \Cref{lem:Phi-lsc}, \(\mathcal E^*\) is sequentially weakly lower
	semicontinuous. It remains to check that this dual energy functional coincides with the Riemannian energy functional:
	
	\begin{lemma}[Identification with the extended Riemannian energy]
		\label{lem:energy-identification}
		Let \(p,q\in\cU\). For $\gamma\in H^1_{p,q}([0,1],\cH)$,
		\begin{equation}\label{eq:extended-energy}
			\mathcal E^*(\gamma)=
			\begin{cases}
				\displaystyle
				\int_0^1
				\widetilde G_{\gamma(t)}(\dot\gamma(t),\dot\gamma(t))\,dt,
				&\gamma([0,1])\subset \cU,\\[1.2em]
				+\infty,&\gamma([0,1])\not\subset \cU.
			\end{cases}
		\end{equation}
		Consequently, for every \(p,q\in\cU\), the extended energy of
		$\widetilde G$ is sequentially weakly lower semicontinuous on
		$H^1_{p,q}([0,1],\cH)$.
	\end{lemma}
	
	\begin{proof}
		We use the absolutely continuous representative of \(\gamma\).
		
		Assume first that \(\gamma([0,1])\subset\cU\). For every
		\(\zeta\in L^2([0,1],\cH)\), the identity
		\[
		Q_{\zeta(t)}(\gamma(t))
		=
		\ip{B(\gamma(t))^{-1}\zeta(t)}{\zeta(t)}
		\]
		from \Cref{lem:inverse-usc} and \eqref{eq:fenchel-pointwise} give
		\[
		\Phi_\zeta(\gamma)
		\leq
		\int_0^1
		\ip{B(\gamma(t))\dot\gamma(t)}{\dot\gamma(t)}\,dt.
		\]
		Taking the supremum over \(\zeta\) yields the corresponding upper bound for
		\(\mathcal E^*(\gamma)\).
		
		Since \(\gamma([0,1])\) is compact and \(B\) is continuous, the operator
		norm of \(B(\gamma(t))\) is uniformly bounded in \(t\). Hence
		\[
		\zeta_\gamma(t):=B(\gamma(t))\dot\gamma(t)
		\]
		belongs to \(L^2([0,1],\cH)\). For this choice we have by \eqref{eq:Phi}
		\[
		\Phi_{\zeta_\gamma}(\gamma)
		=
		\int_0^1
		\ip{B(\gamma(t))\dot\gamma(t)}{\dot\gamma(t)}\,dt.
		\]
		Thus by \ref{it:2_local-conformal} in \Cref{prop:conformal-metric} and \eqref{eq:E-star} 
		\[
		\mathcal E^*(\gamma)
		=
		\int_0^1
		\widetilde G_{\gamma(t)}
		\bigl(\dot\gamma(t),\dot\gamma(t)\bigr)\,dt.
		\]
		This proves \eqref{eq:extended-energy} when
		\(\gamma([0,1])\subset\cU\).
		
		Now suppose that \(\gamma([0,1])\not\subset\cU\). Then
		\(\cU\neq\cH\). Since \(\gamma(0)=p\in\cU\), define the first exit time
		\[
		\tau:=\inf\{t\in[0,1]:\gamma(t)\notin\cU\}.
		\]
		Continuity of \(\gamma\) gives \(0<\tau\leq1\) and \(\gamma(t)\in\cU\) for all \(t<\tau\) and \(\gamma(\tau)\notin\cU\). Therefore for $d$ as \ref{it:3_local-conformal} in \Cref{prop:conformal-metric} we have 
		\[
		d(\gamma(s))\longrightarrow0
		\qquad(s\uparrow\tau).
		\]
		For \(0\leq s<\tau\), we define
		\begin{equation}\label{eq:I_s_functional}
			I(s):=
			\int_0^s
			\ip{B(\gamma(t))\dot\gamma(t)}{\dot\gamma(t)}\,dt.
		\end{equation}
        For every \(s<\tau\), the set \(\gamma([0,s])\) is compact and contained in \(\cU\).  Hence it has positive distance from \(\cH\setminus\cU\), and \(B\circ\gamma\) is uniformly bounded on \([0,s]\).  In particular, \(I(s)<\infty\). Thus the function \(I\) is finite and nondecreasing. We claim that
		\begin{equation}\label{eq:partial-energy-diverges}
			I(s)\longrightarrow+\infty
			\qquad(s\uparrow\tau).
		\end{equation}
		Suppose otherwise. Then there exists \(M>0\) such that
		\[
		I(s)\leq M
		\qquad \forall s<\tau.
		\]
		By \eqref{eq:B-boundary} in combination with the definition \(\Tilde{G}\) in \ref{it:2-controlled-construction} of \Cref{prop:conformal-metric} we have:
		\[
		\int_0^s
		\frac{\norm{\dot\gamma(t)}^2}{d(\gamma(t))^2}\,dt
		\leq M
		\qquad \forall s<\tau.
		\]
		Furthermore the distance function \(d\) is \(1\)-Lipschitz. Hence \(d\circ\gamma\) is
		absolutely continuous and
		\[
		|(d\circ\gamma)'(t)|
		\leq\norm{\dot\gamma(t)}
		\]
		for almost every \(t<\tau\). For each \(s<\tau\),
		\(d\circ\gamma\) has a positive minimum on \([0,s]\). Thus
		\(\log(d\circ\gamma)\) is absolutely continuous on \([0,s]\), and
		\[
		\left|(\log(d\circ\gamma))'(t)\right|
		\leq\frac{\norm{\dot\gamma(t)}}{d(\gamma(t))}
		\]
		for almost every \(t\in[0,s]\). Consequently,
		\[
		\left|\log d(\gamma(s))-\log d(p)\right|
		\leq\int_0^s
		\frac{\norm{\dot\gamma(t)}}{d(\gamma(t))}\,dt  \leq\sqrt{s}
		\left(
		\int_0^s
		\frac{\norm{\dot\gamma(t)}^2}{d(\gamma(t))^2}\,dt
		\right)^{1/2} \leq\sqrt{M}.
		\]
		This contradicts \(d(\gamma(s))\to0\) as \(s\uparrow\tau\), and proves
		\eqref{eq:partial-energy-diverges}.
		
		For almost every \(t<\tau\), the pointwise supremum
		\[
		\sup_{z\in\cH}
		\left(
		2\ip{z}{\dot\gamma(t)}
		-Q_z(\gamma(t))
		\right)
		\]
		is attained at \(  z=B(\gamma(t))\dot\gamma(t).\)
		Thus the field 
		\[
		t\longmapsto B(\gamma(t))\dot\gamma(t)
		\]
		is the pointwise maximizing dual field before the first exit time. It may,
		however, fail to belong to \(L^2([0,\tau),\cH)\): as
		\(t\uparrow\tau\), the operator \(B(\gamma(t))\) may become unbounded, so
		\[
		\int_0^\tau
		\norm{B(\gamma(t))\dot\gamma(t)}^2\,dt
		\]
		need not be finite. Hence this full maximizing field need not be an
		admissible competitor in the supremum defining \(\mathcal E^*\).
		
		For every fixed \(s<\tau\), the curve \(\gamma([0,s])\) remains a positive
		distance away from \(\cH\setminus\cU\). Thus \(B(\gamma(t))\) is uniformly
		bounded on \([0,s]\), and the truncated maximizing field is admissible. Set
		\[
		\zeta_s(t):=
		\begin{cases}
			B(\gamma(t))\dot\gamma(t),&0\leq t\leq s,\\
			0,&s<t\leq1.
		\end{cases}
		\]
		Then \(\zeta_s\in L^2([0,1],\cH)\). For almost every \(t\in[0,s]\),
		\[
		Q_{\zeta_s(t)}(\gamma(t))
		=
		\ip{B(\gamma(t))\dot\gamma(t)}{\dot\gamma(t)}.
		\]
		For \(t>s\), both integrands in \eqref{eq:Phi} vanish. Therefore,
		by \eqref{eq:Phi} and \eqref{eq:I_s_functional},
		\[
		\Phi_{\zeta_s}(\gamma)
		=
		2I(s)-I(s)
		=
		I(s).
		\]
		By \eqref{eq:partial-energy-diverges},
		\(\Phi_{\zeta_s}(\gamma)\to+\infty\) as \(s\uparrow\tau\). Therefore
		\(\mathcal E^*(\gamma)=+\infty\), completing the proof of
		\eqref{eq:extended-energy}.
	\end{proof}
	We close this subsection by finishing the proof of \Cref{prop:completeness-criterion}.
	
	\begin{proof}[Proof of \Cref{prop:completeness-criterion}]It remains to check that \(\tilde G\) has the properties \ref{item:abc-a}-\ref{item:abc-c}. We begin with:\\
		
		\noindent\textbf{\Cref{item:abc-a}.}
		By \eqref{eq:B-ge-I},
		\[
		\norm{v}^2
		\leq\widetilde G_x(v,v)
		\qquad \forall (x,v)\in\cU\times\cH.
		\]
		Thus \Cref{item:abc-a} holds with constant \(1\) on every
		\(\widetilde G\)-metric ball.\\
		
		\noindent\textbf{\Cref{item:abc-b}.}
		Fix \(p\in\cU\) and \(R>0\). If \(\cU=\cH\), the assertion is immediate.
		Assume therefore that \(\cU\neq\cH\).
		
		Let \(\sigma\colon[0,1]\to\cU\) be absolutely continuous. By
		\eqref{eq:B-boundary} and the fact that \(d\) is \(1\)-Lipschitz we have
		\[
		\left|\log d(\sigma(1))-\log d(\sigma(0))\right| \leq \int_0^1
		\frac{\norm{\dot\sigma(t)}}{d(\sigma(t))}\,dt \leq  L_{\widetilde G}(\sigma). 
		\]
		Taking the infimum over all curves joining \(p\) to \(x\) gives
		\[
		\left|\log d(x)-\log d(p)\right|
		\leq d_{\widetilde G}(p,x).
		\]
		Consequently,
		\begin{equation}\label{eq:ball-boundary-separation}
			x\in B_{\widetilde G}(p,R)
			\quad\Longrightarrow\quad
			0<e^{-R}d(p)\leq d(x)
		\end{equation}
		
		Let \((x_n)\) be a sequence in \(B_{\widetilde G}(p,R)\) that converges in
		the Hilbert norm to \(x\in\cH\). By continuity of \(d\) and
		\eqref{eq:ball-boundary-separation},
		\[
		0<e^{-R}d(p) \leq \lim_{n\to\infty}d(x_n)= d(x)
		.
		\]
		Hence \(x\in\cU\). Thus the Hilbert-norm closure of every
		\(\widetilde G\)-metric ball is contained in \(\cU\), as required.\\
		
		\textbf{\Cref{item:abc-c}.}
		By \Cref{lem:energy-identification}, for every \(p,q\in\cU\), the
		extended energy of \(\widetilde G\) is sequentially weakly lower
		semicontinuous on \(H^1_{p,q}([0,1],\cH)\). This is precisely
		\Cref{item:abc-c}.
	\end{proof}
	
	\subsubsection{Conclusion of the proof}
	\label{sec:local-conformal-conclusion}
	
	We now combine \Cref{prop:conformal-metric} and
	\Cref{prop:completeness-criterion} to prove
	\Cref{thm:local-conformal}.
	
	\begin{proof}[Proof of \Cref{thm:local-conformal}]
		By \Cref{prop:conformal-metric}, there exists a smooth strong Riemannian
		metric \(\widetilde G=\rho G\) conformal to \(G\). By
		\Cref{prop:completeness-criterion}, \(\widetilde G\) satisfies the
		hypotheses \ref{item:abc-a}-\ref{item:abc-c} of \Cref{thm:abstract_completeness}. Thus we can apply \Cref{thm:abstract_completeness}, which therefore yields
		metric completeness of $(\cU,d_{\Tilde{G}})$, geodesic completeness of $(\cU,\tilde G)$, and the existence of a length-minimizing geodesic of $(\cU,\tilde G)$
		between any two points in the same connected component of \(\cU\).
	\end{proof}
	\section{Local Conformal Flexibility of Completeness}
\label{s:local-conformal-flexibility}

The conformal metric constructed in \Cref{thm:main-theorem} is metrically
and geodesically complete, and any two points in the same connected component
are joined by a length-minimizing geodesic.  We now explore the local
conformal flexibility of metric and geodesic completeness within a fixed
strong conformal class.

\begin{theorem}[=\Cref{thm:intro-local-atkin}]
	\label{thm:local-conformal-atkin}
	Let \(\cH\) be an infinite-dimensional separable real Hilbert space, let
	\(\cU\subset\cH\) be a nonempty open subset, and let \(G\) be a metrically
	complete smooth strong Riemannian metric on \(\cU\). For every Hilbert-norm
	ball \(W\subset\cH\) with \(\overline W\subset\cU\), where the closure is
	taken in \(\cH\), there is a smooth function
	\[
	\sigma_{\mathrm{Atkin}}\colon\cU\longrightarrow(0,1],
	\qquad
	\sigma_{\mathrm{Atkin}}=1\quad\text{on }\cU\setminus W,
	\]
	such that \(\sigma_{\mathrm{Atkin}}G\) is geodesically complete but not
	metrically complete.
\end{theorem}
\begin{remark}[Relation to Atkin's construction]
	Atkin's construction~\cite{Atkin97} establishes the existence of a strong
	metric on a Hilbert space that is geodesically complete but metrically
	incomplete. \Cref{thm:local-conformal-atkin} shows that this separation is
	locally flexible: starting from any metrically complete strong metric \(G\),
	it can be achieved within the conformal class of \(G\) by a deformation that
	leaves \(G\) unchanged outside any prescribed Hilbert-norm ball \(W\) with
	\(\overline W\subset\cU\).
\end{remark}

Combined with \Cref{thm:main-theorem}, the local theorem has the following
global consequence.

\begin{corollary}[Conformal nonclosedness of metric completeness]
	\label{thm:atkin-deformation}
	Let \(\cM\) be an infinite-dimensional separable Hilbert manifold equipped
	with a smooth strong Riemannian metric \(G\). Then there exists a smooth path
	\[
	\rho\colon[0,1]\times\cM\longrightarrow(0,\infty)
	\]
	such that \(\rho(t,\cdot)G\) is metrically complete for every
	\(0\leq t<1\), the metric \(\rho(0,\cdot)G\) satisfies the conclusions of
	\Cref{thm:main-theorem}, and \(\rho(1,\cdot)G\) is geodesically complete
	but not metrically complete.
\end{corollary}
\begin{remark}[Metric completeness is not closed in the compact-open topology.]
    In particular, metric completeness is not closed in the compact-open
\(C^\infty\) topology on smooth positive conformal factors.
\end{remark}
\begin{remark}[Nonclosedness of geodesic completeness]
	\label{rem:geodesic-nonclosedness}
	Under the same hypotheses, there is also a smooth conformal path whose
	metrics are metrically complete for every \(t<1\), whose initial metric
	satisfies the conclusions of \Cref{thm:main-theorem}, and whose endpoint is
	geodesically incomplete. At a high level, one first chooses a closed
	embedded escaping curve inside a Hilbert-norm ball and prescribes the first
	jet of a positive conformal factor along that curve so that the curve
	becomes a finite-time geodesic for the resulting conformal metric. A smooth
	extension of this prescribed first jet can be chosen to equal one away from
	the chosen ball. Interpolating this endpoint factor with the complete
	initial factor gives the asserted path.
\end{remark}
\begin{proof}[Proof of \Cref{thm:atkin-deformation}]
	By \Cref{thm:main-theorem}, choose a conformal factor
	\(\rho_{\mathrm{HR}}\) such that
	\[
	G_{\mathrm{HR}}:=\rho_{\mathrm{HR}}G
	\]
	is metrically and geodesically complete and admits minimizing geodesics
	between arbitrary endpoints.

	Choose an Eells--Elworthy diffeomorphism
	\(\Phi\colon\cM\to\cU\) onto an open subset of a separable Hilbert space,
	and set
	\[
	G_{\cU}:=(\Phi^{-1})^*G_{\mathrm{HR}}.
	\]
	Choose a Hilbert-norm ball \(W\) with \(\overline W\subset\cU\).
	By \Cref{thm:local-conformal-atkin}, there is a smooth function
	\(\sigma_{\cU}\colon\cU\to(0,1]\), equal to one outside \(W\), such that
	\(\sigma_{\cU}G_{\cU}\) is geodesically complete but not metrically
	complete.  Define
	\[
	\sigma:=\sigma_{\cU}\circ\Phi.
	\]
	Then \(\sigma G_{\mathrm{HR}}\) is geodesically complete but not metrically
	complete.

	Now set
	\[
	\rho(t,x):=\rho_{\mathrm{HR}}(x)
	\bigl((1-t)+t\sigma(x)\bigr).
	\]
	Then \(\rho\) is smooth and positive, and
	\[
	\rho(0,\cdot)G=G_{\mathrm{HR}},
	\qquad
	\rho(1,\cdot)G=\sigma G_{\mathrm{HR}}.
	\]
	For \(0\leq t<1\), writing \(G_t:=\rho(t,\cdot)G\), one has
	\[
	(1-t)G_{\mathrm{HR}}\leq G_t\leq G_{\mathrm{HR}}.
	\]
	Hence \(G_t\) is uniformly equivalent to the metrically complete metric
	\(G_{\mathrm{HR}}\), and is therefore metrically complete.
\end{proof}
The remainder of this section is devoted to the proof of
\Cref{thm:local-conformal-atkin}. The proof has three genuinely nontrivial tasks.  First, Atkin's original end
is global, whereas here it must be implanted inside an arbitrarily prescribed
bounded ball without acquiring an ambient accumulation point.  The
construction in \Cref{prop:localized-atkin-end} achieves this by encoding the
unbounded escape parameter in translated bumps; it produces uniform tubular
coordinates, a localized conformal factor, and a finite-length escaping curve
\(\beta\).  The latter gives metric incompleteness.

Second, this construction must work for an arbitrary complete strong
background metric rather than only for the flat Hilbert metric.  The local
comparison estimates in the proof of \Cref{thm:local-conformal-atkin} transfer
the uniform geometric control of the localized model to the given metric.

Finally, a finite-length escaping curve does not itself say anything about
geodesic completeness.  Since the conformal factor degenerates along the
implanted end, direct comparison with the complete background metric is
insufficient.  The auxiliary escape coordinate \(Q\) and vector field \(Y\)
from \Cref{prop:localized-atkin-end} satisfy the virial inequality
\eqref{eq:localized-end-virial}.  \Cref{prop:virial-escape-criterion} turns
this into a dynamical obstruction: a hypothetical finite-time escaping
geodesic would force its \(Y\)-momentum to become unbounded, contradicting
the energy bound.  This yields geodesic completeness despite metric
incompleteness.
	\subsection{Proof of
		\texorpdfstring{\Cref{thm:local-conformal-atkin}}
		{the localized conformal separation theorem}}
	\label{s:proof-local-conformal-atkin}
	
	We now carry out these three steps.  The construction of the localized model
is established in \Cref{prop:localized-atkin-end}, whose proof is deferred to
Sections~\ref{app:localized-virial-end}--\ref{app:localized-atkin-end-proof}.
Geodesic completeness then follows from the escape criterion in
\Cref{prop:virial-escape-criterion}.
	
	\subsubsection{Construction of the localized model}
	
	The geometry behind the construction is simple. If
	\(q(x):=\sqrt{1+\lVert x\rVert^2}\), the map
	\[
	x\longmapsto
	\lambda\left(
	\frac{x}{q(x)},\,h_0(\,\cdot-q(x))
	\right)
	\]
	compresses the Atkin model into a bounded Hilbert ball while recording the
	unbounded variable \(q(x)\) in a translated bump. The translates remain
	bounded in norm but have no convergent subsequence as \(q\to\infty\).
	Uniform tubular coordinates then allow the conformal factor and the
	vector field entering the virial estimate to be extended to a neighborhood
	of this embedded end.
	
	\begin{proposition}[Localized Atkin end]
		\label{prop:localized-atkin-end}
		Let \(\cH\) be an infinite-dimensional separable real Hilbert space and let
		\(r>0\). Then there are \(q_0>1\), constants
		\(\delta,\Lambda_0>0\), an open set \(\cV\subset B_{\cH}(0,r),\)
		and smooth maps
		\[
		Q\colon\cV\to(q_0,\infty),
		\qquad
		Y\colon\cV\to\cH,
		\]
		a function \(\vartheta\in C^\infty(\cH,[0,1])\), a positive smooth function
		\(h\colon[q_0,\infty)\to(0,\infty)\), and a piecewise smooth curve
		\(\beta\colon[0,\infty)\to\cV\) with the following properties:
		\begin{enumerate}[label=\textup{(\roman*)}]
			\item \label{it:localized-end-geometry}
			\emph{(Controlled geometry.)}
			One has \(\vartheta=0\) on \(\cH\setminus\cV\),
			\[
			\sup_{\cH}\lVert d\vartheta\rVert
			+\sup_{\cV}
			\bigl(\lVert dQ\rVert+\lVert Y\rVert+\lVert dY\rVert\bigr)
			<\infty,
			\qquad
			\int_{q_0}^{\infty}\sqrt{h(q)}\,dq=\infty.
			\]
			
			\item \label{it:localized-end-factor}
			\emph{(Localized conformal factor.)}
			For every \(\Lambda\geq\Lambda_0\), there is
			\(\rho_\Lambda\in C^\infty(\cH,(0,\Lambda])\) such that
			\[
			\rho_\Lambda=\Lambda\quad\text{on }\cH\setminus\cV,
			\qquad
			\rho_\Lambda\geq\Lambda(1-\vartheta),
			\]
			and, on \(\cV\),
			\begin{equation}\label{eq:localized-end-virial}
				\frac12\,d\rho_\Lambda(Y)+\delta\rho_\Lambda\geq h\circ Q.
			\end{equation}
			Moreover, for every \(R>q_0\), there is \(m_R>0\) such that
			\[
			\rho_\Lambda(x)\geq m_R
			\quad\text{whenever}\quad
			\vartheta(x)\geq\frac14,\quad Q(x)\leq R.
			\]
			
			\item \label{it:localized-end-escape}
			\emph{(Metric escape.)}
			The curve \(\beta\) lies in \(\{\vartheta=1\}\), has finite
			\(\rho_\Lambda\langle\cdot,\cdot\rangle_{\cH}\)-length for every
			\(\Lambda\geq\Lambda_0\), and, for every sequence \(t_n\to\infty\),
			the sequence \((\beta(t_n))_n\) has no convergent subsequence in
			\(\cH\).
		\end{enumerate}
	\end{proposition}
	
	The three items have distinct roles. The first gives uniform control of the
	localized end. The second contains the virial estimate that will prevent
	finite-time geodesic escape. The third supplies the curve witnessing metric
	incompleteness. The proof is deferred to
	\Cref{app:localized-virial-end}.
	
	\subsubsection{A virial escape criterion}
	
	We isolate the completeness argument from the particular Atkin
	construction. This makes the mechanism visible without carrying the
	tubular coordinates through the proof.
	
	\begin{proposition}[Virial escape criterion]
		\label{prop:virial-escape-criterion}
		Let \(G\) be a metrically complete smooth strong Riemannian metric on an
		open subset \(\cU\) of a Hilbert space \(\cH\). Fix \(\Lambda>0\), and
		let
		\[
		\rho\in C^\infty(\cU,(0,\Lambda]).
		\]
		Suppose
		there are an open set \(\cV\subset\cU\), a function
		\(\vartheta\in C^\infty(\cU,[0,1])\), smooth maps
		\[
		Q\colon\cV\to(q_0,\infty),
		\qquad
		Y\colon\cV\to\cH,
		\]
		a positive continuous function
		\[
		h\colon[q_0,\infty)\to(0,\infty),
		\]
		and positive constants \(\delta,K,C_\vartheta,C_Q,C_Y\) such that:
		\begin{enumerate}[label=\textup{(\roman*)}]
			\item
			\emph{(Cutoff barrier.)}
			\(\vartheta=0\) on \(\cU\setminus\cV\),
			\(\rho\geq\Lambda(1-\vartheta)\), and
			\[
			\lvert d\vartheta\rvert_{G^*}\leq C_\vartheta;
			\]
			
			\item
			\emph{(Virial control.)}
			On \(\cV\),
			\[
			\lvert dQ\rvert_{G^*}\leq C_Q,
			\qquad
			\lVert Y\rVert_G\leq C_Y,
			\qquad
			\frac12\mathcal L_YG\geq-KG,
			\]
			and
			\[
			\frac12\,d\rho(Y)+\delta\rho\geq h\circ Q,
			\qquad
			\int_{q_0}^{\infty}\sqrt{h(q)}\,dq=\infty;
			\]
			
			\item
			\emph{(No stagnation at finite \(Q\).)}
			For every \(R>q_0\), there is \(m_R>0\) such that
			\[
			\rho(x)\geq m_R
			\quad\text{whenever}\quad
			\vartheta(x)\geq\frac14,\quad Q(x)\leq R.
			\]
		\end{enumerate}
		Then \(\rho G\) is geodesically complete.
	\end{proposition}
	
	The proof, given in \Cref{app:virial-escape-criterion}, is a contradiction
	argument. In mechanical time, a hypothetical finite-time geodesic satisfies
	\[
	\int\rho\,ds<\infty.
	\]
	The cutoff barrier traps its tail in \(\cV\), where \(Q\) must become
	unbounded. The divergence of
	\(\int^\infty\sqrt{h(q)}\,dq\) then implies
	\(\int h(Q)\,ds=\infty\). For the momentum \(J:=G(v,Y)\), one has
	\[
	J'\geq\mu h(Q)-\mu(\delta+K)\rho,
	\qquad
	|J|\leq C_Y\sqrt{\mu\Lambda}.
	\]
	Thus the first inequality forces \(J\) to be unbounded, whereas the second
	keeps it bounded.
	
	\subsubsection{Conclusion of the proof}
	
	\begin{proof}[Proof of \Cref{thm:local-conformal-atkin}]
		Choose a Hilbert-norm ball \(B=B_{\cH}(x_*,r)\) with
		\(\overline B\subset W\), and translate the coordinates so that \(x_*=0\).
		After decreasing \(r\), there are constants \(m,M,D>0\) such that
		\begin{equation}\label{eq:atkin-background-comparison}
			m\lVert v\rVert_{\cH}^2
			\leq G_x(v,v)
			\leq M\lVert v\rVert_{\cH}^2,
			\qquad
			\lVert dG_x\rVert_{\mathrm{op}}\leq D
		\end{equation}
		for all \(x\in B\) and \(v\in\cH\).
		
		Apply \Cref{prop:localized-atkin-end} with this \(r\), and choose
	\(\Lambda\geq\Lambda_0\) and set \(\rho:=\rho_\Lambda\), where $\Lambda$ and $\rho_\Lambda$ are as in \Cref{prop:localized-atkin-end}.
		
		\smallskip
		\noindent\emph{Metric incompleteness.}
		By \ref{it:localized-end-escape} of \Cref{prop:localized-atkin-end} and the upper estimate in
		\eqref{eq:atkin-background-comparison},
		\[
		L_{\rho G}(\beta)
		\leq
		\sqrt M\,
		L_{\rho\langle\cdot,\cdot\rangle_{\cH}}(\beta)
		<\infty.
		\]
		Choosing points along the tail of \(\beta\) with vanishing remaining length
		produces a \(d_{\rho G}\)-Cauchy sequence. It has no limit, since
		\((\beta(t_n))_n\) has no norm-convergent subsequence and the distance of a
		smooth strong metric induces the manifold topology. Thus \(\rho G\) is not
		metrically complete.
		
		\smallskip
		\noindent\emph{Geodesic completeness.}
		The bounds in \ref{it:localized-end-geometry} in \Cref{prop:localized-atkin-end}, together with
		\eqref{eq:atkin-background-comparison}, provide constants
		\(C_\vartheta,C_Q,C_Y,K>0\) such that
		\[
		\lvert d\vartheta\rvert_{G^*}\leq C_\vartheta,
		\qquad
		\lvert dQ\rvert_{G^*}\leq C_Q,
		\qquad
		\lVert Y\rVert_G\leq C_Y,
		\qquad
		\frac12\mathcal L_YG\geq-KG.
		\]
		Hence \ref{it:localized-end-factor} verifies all remaining hypotheses of
		\Cref{prop:virial-escape-criterion}, and \(\rho G\) is geodesically
		complete.
		
		Finally, define
		\[
		\sigma_{\mathrm{Atkin}}:=\Lambda^{-1}\rho.
		\]
		Then \(0<\sigma_{\mathrm{Atkin}}\leq1\) and
		\(\sigma_{\mathrm{Atkin}}=1\) on
		\(\cU\setminus\cV\), hence on \(\cU\setminus W\). Constant rescaling
		preserves both completeness properties, so
		\(\sigma_{\mathrm{Atkin}}G\) is geodesically complete but not metrically
		complete.
	\end{proof}
	\subsection{Proofs of the propositions}

	\subsubsection{The abstract escape criterion}
	\label{app:virial-escape-criterion}
	
	This subsection contains the momentum argument underlying geodesic
	completeness in \Cref{thm:local-conformal-atkin}.
	
	\begin{proof}[Proof of
		\Cref{prop:virial-escape-criterion}]
		Suppose that \(\gamma\colon[0,T)\to\cU\) is a nonconstant maximal
		\(\rho G\)-geodesic with \(T<\infty\). Its energy
		\[
		\mu:=\rho(\gamma(t))
		G_{\gamma(t)}(\dot\gamma(t),\dot\gamma(t))
		\]
		is a positive constant.
		
		We first show that
		\begin{equation}\label{eq:virial-criterion-liminf}
			\liminf_{t\uparrow T}\rho(\gamma(t))=0.
		\end{equation}
	Otherwise, there are \(\epsilon>0\) and \(t_0<T\) such that
	\(\rho(\gamma(t))\geq\epsilon\) for all \(t\in[t_0,T)\). Since the
	\(\rho G\)-energy is the constant \(\mu\),
	\[
	L_G\bigl(\gamma|_{[t_0,T)}\bigr)
	\leq\sqrt{\mu/\epsilon}\,(T-t_0)<\infty.
	\]
	Hence \(\gamma(t)\) is a \(d_G\)-Cauchy curve as \(t\uparrow T\). Metric
	completeness of \(G\) yields a limit \(x_\infty\in\cU\). Choose a coordinate
	chart about \(x_\infty\) and a Hilbert-norm ball centered at its coordinate
	image whose closure is contained in the chart image. On the inverse image of
	a smaller such ball, the smooth strong metrics \(G\) and \(\rho G\) are
	uniformly equivalent to the Hilbert norm, and the Christoffel map of
	\(\rho G\) is bounded. Since \(\gamma(t)\) converges to \(x_\infty\), it lies
	in this smaller neighborhood for all sufficiently large \(t\). The
	constant-energy identity therefore gives a uniform bound for the coordinate
	velocity \(\dot\gamma\), and the geodesic equation gives a uniform bound for
	the coordinate acceleration. Thus \(\dot\gamma(t)\) is Cauchy as
	\(t\uparrow T\), so the limiting position and velocity provide initial data
	for the local geodesic equation at time \(T\). Local existence extends
	\(\gamma\) beyond \(T\), contradicting maximality.
		
		Since \(\rho\geq\Lambda(1-\vartheta)\),
		\eqref{eq:virial-criterion-liminf} implies that \(\gamma\) enters
		\(\{\vartheta\geq3/4\}\) arbitrarily late. Every curve crossing from this
		set to \(\{\vartheta\leq1/4\}\) has \(\rho G\)-length at least
		\[
		b:=\frac{\sqrt\Lambda}{4C_\vartheta}>0.
		\]
		Indeed, on the transition band
		\(1/4\leq\vartheta\leq3/4\), one has \(\rho\geq\Lambda/4\), and the total
		variation of \(\vartheta\) along a crossing is at least \(1/2\). Choose
		\(t_0<T\) with \(\sqrt\mu(T-t_0)<b\), and then \(t_1>t_0\) such that
		\(\vartheta(\gamma(t_1))\geq3/4\). The crossing estimate shows that
		\[
		\vartheta(\gamma(t))>\frac14
		\qquad(t\geq t_1).
		\]
		In particular, this tail lies in \(\cV\). If \(Q\) were bounded there,
		assumption~\textup{(iii)} would contradict
		\eqref{eq:virial-criterion-liminf}. Thus \(Q\) is unbounded on the tail.
		
		Introduce mechanical time by
		\[
		ds=\frac{dt}{\rho(\gamma(t))},
		\qquad
		v:=\frac{d\gamma}{ds},
		\]
		and denote its endpoint by \(S_*\in(0,\infty]\). Then
		\begin{equation}\label{eq:virial-criterion-mechanical}
			G(v,v)=\mu\rho,
			\qquad
			\int_0^{S_*}\rho(\gamma(s))\,ds=T<\infty.
		\end{equation}
		Let \(s_0\) correspond to \(t_1\). From the bound on \(dQ\),
		\[
		|Q'|\leq C_Q\sqrt{\mu\rho}.
		\]
		Set
		\[
		\mathcal H(q):=\int_{q_0}^q\sqrt{h(r)}\,dr.
		\]
		For \(s\geq s_0\), Cauchy--Schwarz gives
		\[
		\begin{aligned}
			\left|\mathcal H(Q(s))-\mathcal H(Q(s_0))\right|
			&\leq C_Q\sqrt\mu
			\int_{s_0}^s\sqrt{h(Q)\rho}\,ds\\
			&\leq C_Q\sqrt\mu
			\left(\int_{s_0}^s h(Q)\,ds\right)^{1/2}
			\left(\int_{s_0}^s\rho\,ds\right)^{1/2}.
		\end{aligned}
		\]
		Since \(Q\) is continuous and unbounded, there are
		\(s_j\uparrow S_*\) with \(Q(s_j)\to\infty\). The divergence of
		\(\mathcal H\), together with
		\eqref{eq:virial-criterion-mechanical}, therefore yields
		\begin{equation}\label{eq:virial-criterion-radial-divergence}
			\int_{s_0}^{S_*}h(Q(s))\,ds=\infty.
		\end{equation}
		
		On the mechanical tail, define \(J:=G(v,Y)\). The conformal geodesic
		equation gives
		\[
		J'
		=\frac{\mu}{2}\,d\rho(Y)
		+\frac12(\mathcal L_YG)(v,v).
		\]
		Using the virial and Lie-derivative estimates together with
		\(G(v,v)=\mu\rho\), we obtain
		\[
		J'\geq\mu h(Q)-\mu(\delta+K)\rho.
		\]
		After integration, the right-hand side tends to \(+\infty\) by
		\eqref{eq:virial-criterion-mechanical} and
		\eqref{eq:virial-criterion-radial-divergence}. Hence \(J\) is unbounded.
		On the other hand,
		\[
		|J|\leq\sqrt{\mu\rho}\,\lVert Y\rVert_G
		\leq C_Y\sqrt{\mu\Lambda},
		\]
	a contradiction.
	
	Applying the same argument to the time-reversed geodesic excludes a finite
	left endpoint as well. Hence \(\rho G\) is geodesically complete.
\end{proof}

\subsubsection{Construction of the localized Atkin end}
	\label{app:localized-virial-end}
	
	This appendix proves \Cref{prop:localized-atkin-end}. The construction has
	three stages. We first extract radial data and a virial estimate from
	Atkin's model. We then construct uniform tubular coordinates around a curve
	of translated bumps. Finally, we combine these ingredients to place the
	Atkin end in an arbitrarily small Hilbert ball.
	
	\paragraph{\textbf{Radialized Atkin data}}
	
	For the convenience of the reader, we recall only the ingredients from
	Atkin's construction~\cite{Atkin97} that will be used below. Set
	\[
	\Psi(r):=(e^e+r)\log^2(e^e+r).
	\]
	Let
	\[
	\bigl(\cH_{\mathrm A},
	\langle\cdot,\cdot\rangle_{\cH_{\mathrm A}}\bigr)
	\]
	and \(A,C,Z,S\) be the data constructed in
	\cite[Proposition~5.21]{Atkin97}. Thus \(A,C,Z\) are bounded nonnegative
	self-adjoint operators and \(S\) is bounded and skew-adjoint. The conformal
	factor in Atkin's model~\cite[(3.1)]{Atkin97} is
	\begin{equation}\label{eq:f_A_tilde}
		\widetilde f_{\mathrm A}\colon
		\cH_{\mathrm A}\longrightarrow(0,\infty),
		\qquad
		x\longmapsto
		\frac{\langle Ax,x\rangle_{\cH_{\mathrm A}}
			+\exp\bigl(-\langle Cx,x\rangle_{\cH_{\mathrm A}}\bigr)}
		{\Psi\bigl(\langle Zx,x\rangle_{\cH_{\mathrm A}}\bigr)}.
	\end{equation}
	With the convention \([T,S]:=TS-ST\), the three relations from
	\cite[(4.6)(i)--(iii)]{Atkin97} used below are
	\begin{equation}\label{eq:atkin-commutators}
		[A,S]\geq0,
		\qquad
		\eta_0C\leq[A,S],
		\qquad
		[C,S]\leq\eta_1[A,S]+\nu_1A,
	\end{equation}
	where \(\eta_0,\nu_1>0\) and \(\eta_1\in(0,1)\).
	
	We replace the quadratic form
	\(\langle Zx,x\rangle_{\cH_{\mathrm A}}\) in
	\eqref{eq:f_A_tilde} by a scalar multiple of the squared Hilbert norm.
	Choose \(a_{\mathrm{rad}}>0\) such that
	\[
	a_{\mathrm{rad}}I_{\cH_{\mathrm A}}\geq Z,
	\]
	and define
	\begin{equation}\label{eq:atkin-factor}
		f_{\mathrm A}\colon\cH_{\mathrm A}\longrightarrow(0,\infty),
		\qquad
		x\longmapsto
		\frac{\langle Ax,x\rangle_{\cH_{\mathrm A}}
			+\exp\bigl(-\langle Cx,x\rangle_{\cH_{\mathrm A}}\bigr)}
		{\Psi\bigl(a_{\mathrm{rad}}
			\lVert x\rVert_{\cH_{\mathrm A}}^2\bigr)}.
	\end{equation}
	Write
	\[
	\widetilde G_{\mathrm A}
	:=\widetilde f_{\mathrm A}
	\langle\cdot,\cdot\rangle_{\cH_{\mathrm A}},
	\qquad
	G_{\mathrm A}
	:=f_{\mathrm A}
	\langle\cdot,\cdot\rangle_{\cH_{\mathrm A}}.
	\]
	The properties needed below are recorded in the following lemma.
	
	\begin{lemma}[Radialized Atkin factor]
		\label{lem:atkin-radial-data}
		The function \(f_{\mathrm A}\) is smooth, bounded, and strictly positive.
		Moreover:
		\begin{enumerate}[label=\textup{(\roman*)}]
			\item\label{it:atkin-finite-curve}
			there is a piecewise smooth curve
			\(\alpha\colon[0,\infty)\to\cH_{\mathrm A}\) such that
			\begin{equation}\label{eq:atkin-finite-curve}
				L_{G_{\mathrm A}}(\alpha)<\infty,
				\qquad
				\lVert\alpha(t)\rVert_{\cH_{\mathrm A}}\longrightarrow\infty;
			\end{equation}
			
			\item\label{it:atkin-virial}
			there are constants \(\delta_{\mathrm A},c_{\mathrm A}>0\) such that
			\begin{equation}\label{eq:atkin-virial}
				\frac12\,d(f_{\mathrm A})_x(Sx)
				+\delta_{\mathrm A}f_{\mathrm A}(x)
				\geq
				\frac{c_{\mathrm A}}
				{\Psi\bigl(a_{\mathrm{rad}}
					\lVert x\rVert_{\cH_{\mathrm A}}^2\bigr)}
				\qquad (x\in\cH_{\mathrm A}).
			\end{equation}
		\end{enumerate}
	\end{lemma}
	
	\begin{proof}
		Smoothness and strict positivity follow immediately from
		\eqref{eq:atkin-factor}. If
		\(R=\lVert x\rVert_{\cH_{\mathrm A}}^2\), then
		\[
		0<f_{\mathrm A}(x)
		\leq
		\frac{\lVert A\rVert_{\mathrm{op}}R+1}
		{\Psi(a_{\mathrm{rad}}R)}.
		\]
		The expression on the right is bounded for \(R\geq0\).
		
		\textbf{\Cref{it:atkin-finite-curve}.}
		The proof of \cite[Proposition~5.20]{Atkin97} provides a polygonal path
		\(\alpha\) of finite \(\widetilde G_{\mathrm A}\)-length. Since \(\Psi\)
		is increasing and \(a_{\mathrm{rad}}I_{\cH_{\mathrm A}}\geq Z\),
		\[
		f_{\mathrm A}\leq\widetilde f_{\mathrm A},
		\qquad
		G_{\mathrm A}\leq\widetilde G_{\mathrm A}.
		\]
		Consequently,
		\[
		L_{G_{\mathrm A}}(\alpha)
		\leq L_{\widetilde G_{\mathrm A}}(\alpha)<\infty.
		\]
		
		Let \(e_n\) be the vertices of the polygonal tail. Their construction and
		the relevant estimates in \cite[(5.17)--(5.18)]{Atkin97} give, after
		discarding finitely many initial vertices,
		\[
		e_n\perp e_{n+1},
		\qquad
		\lVert e_{n+1}\rVert_{\cH_{\mathrm A}}
		>2\lVert e_n\rVert_{\cH_{\mathrm A}},
		\qquad
		\lVert e_n\rVert_{\cH_{\mathrm A}}\longrightarrow\infty.
		\]
		Hence, for \(0\leq t\leq1\),
		\[
		\bigl\lVert(1-t)e_n+te_{n+1}\bigr\rVert_{\cH_{\mathrm A}}^2
		=(1-t)^2\lVert e_n\rVert_{\cH_{\mathrm A}}^2
		+t^2\lVert e_{n+1}\rVert_{\cH_{\mathrm A}}^2\geq \frac14\lVert e_n\rVert_{\cH_{\mathrm A}}^2.
		\]
		Thus the entire polygonal tail escapes in norm, proving
		\ref{it:atkin-finite-curve}.
		
		\textbf{\Cref{it:atkin-virial}.}
		Put
		\[
		r:=\langle Cx,x\rangle_{\cH_{\mathrm A}},
		\qquad
		D(x):=\Psi\bigl(
		a_{\mathrm{rad}}\lVert x\rVert_{\cH_{\mathrm A}}^2
		\bigr).
		\]
		Since \(S\) is skew-adjoint,
		\[
		d\bigl(\lVert\cdot\rVert_{\cH_{\mathrm A}}^2\bigr)_x(Sx)
		=2\langle x,Sx\rangle_{\cH_{\mathrm A}}=0.
		\]
		Therefore \(D\) is invariant in the direction \(Sx\), and
		\[
		\frac12\,d(f_{\mathrm A})_x(Sx)
		=
		\frac{
			\frac12\langle[A,S]x,x\rangle_{\cH_{\mathrm A}}
			-\frac12e^{-r}\langle[C,S]x,x\rangle_{\cH_{\mathrm A}}
		}{D(x)}.
		\]
		Choose \(\delta_{\mathrm A}\geq\nu_1/2\). From
		\eqref{eq:atkin-commutators},
		\[
		\begin{aligned}
			&\frac12\,d(f_{\mathrm A})_x(Sx)
			+\delta_{\mathrm A}f_{\mathrm A}(x) \\
			&\quad\geq
			\frac{
				\frac12(1-\eta_1e^{-r})
				\langle[A,S]x,x\rangle_{\cH_{\mathrm A}}
				+\left(\delta_{\mathrm A}-\frac{\nu_1}{2}e^{-r}\right)
				\langle Ax,x\rangle_{\cH_{\mathrm A}}
				+\delta_{\mathrm A}e^{-r}
			}{D(x)} \\
			&\quad\geq
			\frac{
				\frac12(1-\eta_1)
				\langle[A,S]x,x\rangle_{\cH_{\mathrm A}}
				+\delta_{\mathrm A}e^{-r}
			}{D(x)} \\
			&\quad\geq
			\frac{
				\frac12(1-\eta_1)\eta_0r+\delta_{\mathrm A}e^{-r}
			}{D(x)}.
		\end{aligned}
		\]
		The numerator in the last expression has a strictly positive minimum on
		\([0,\infty)\). Taking \(c_{\mathrm A}\) equal to this minimum proves
		\ref{it:atkin-virial}.
	\end{proof}
	
	\paragraph{\textbf{Uniform tubular coordinates for translated bumps}}
	\label{app:translated-bump-tube}
	
	The translated-bump curve records an unbounded parameter inside a bounded
	set. The following lemma supplies uniform coordinates around that curve and
	ensures that its infinite end has no ambient accumulation point.
	
	\begin{lemma}[Uniform tube for translated bumps]
		\label{lem:translated-bump-tube}
		Let \(0\ne h\in C^\infty_c(\bR)\), put \(\cK:=L^2(\bR)\), and set
		\(k(q):=h(\,\cdot-q)\). For every \(\lambda>0\) and \(q_*\in\bR\),
		there are a closed hyperplane \(N\subset\cK\), a smooth map
		\(U\colon\bR\to\mathcal L(\cK)\) taking values in the orthogonal
		operators, and \(\varepsilon>0\) such that, with
		\[
		\Theta_\lambda\colon
		\bR\times B_N(0,\varepsilon)\longrightarrow\cK,
		\qquad
		\Theta_\lambda(q,z):=\lambda k(q)+U(q)z,
		\]
		the following properties hold:
		\begin{enumerate}[label=\textup{(\roman*)}]
			\item\label{it:bump-tube-normal}
			\(U(q)N=(k'(q))^\perp\) for every \(q\in\bR\);
			
			\item\label{it:bump-tube-coordinates}
			\(\Theta_\lambda\) is a diffeomorphism onto an open subset of \(\cK\),
			and
			\[
			\sup_{(q,z)\in\bR\times B_N(0,\varepsilon)}
			\left(
			\lVert d\Theta_\lambda(q,z)\rVert
			+\bigl\lVert d\Theta_\lambda(q,z)^{-1}\bigr\rVert
			\right)<\infty;
			\]
			
			\item\label{it:bump-tube-separation}
			if \(q_n\to\infty\) and
			\(z_n\in B_N(0,\varepsilon)\), then
			\(\Theta_\lambda(q_n,z_n)\) has no convergent subsequence in \(\cK\).
		\end{enumerate}
	\end{lemma}
	
	\begin{proof}
		\textbf{\Cref{it:bump-tube-normal}.}
		Put
		\[
		\tau(q):=\frac{k'(q)}{\lVert h'\rVert_{L^2}}.
		\]
		The denominator is nonzero, since a nonzero compactly supported smooth
		function cannot be constant. Thus \(\tau\) is a smooth unit vector field
		whose derivatives are uniformly bounded. For \(u,v\in\cK\), write
		\((u\otimes v)w:=\langle w,v\rangle_{\cK}u\), and define
		\[
		B(q):=\tau'(q)\otimes\tau(q)-\tau(q)\otimes\tau'(q).
		\]
		Then
		\[
		B(q)^*=-B(q),
		\qquad
		B(q)\tau(q)=\tau'(q),
		\]
		and \(B\), together with all its derivatives, is uniformly bounded in
		operator norm. Let \(U\) solve
		\[
		U'(q)=B(q)U(q),
		\qquad
		U(q_*)=\id_{\cK}.
		\]
		The map \(U\) is smooth in the operator norm and takes values in the
		orthogonal operators. Both \(U(q)\tau(q_*)\) and \(\tau(q)\) solve the
		same initial-value problem, and hence
		\[
		U(q)\tau(q_*)=\tau(q).
		\]
		Taking \(N:=\tau(q_*)^\perp\), we obtain
		\[
		U(q)N=\tau(q)^\perp=(k'(q))^\perp.
		\]
		
		\textbf{\Cref{it:bump-tube-coordinates}.}
		At \(z=0\),
		\[
		d\Theta_\lambda(q,0)(a,\zeta)
		=a\lambda k'(q)+U(q)\zeta,
		\]
		and the two summands are orthogonal. With
		\[
		b:=\min\{\lambda\lVert h'\rVert_{L^2},1\},
		\]
		it follows that
		\[
		\bigl\lVert d\Theta_\lambda(q,0)(a,\zeta)\bigr\rVert
		\geq b\sqrt{a^2+\lVert\zeta\rVert^2}.
		\]
		Moreover,
		\[
		\bigl\|
		\bigl(d\Theta_\lambda(q,z)-d\Theta_\lambda(q,0)\bigr)(a,\zeta)
		\bigr\|
		\leq\lVert U'(q)\rVert\lVert z\rVert |a|.
		\]
		We may therefore choose \(\varepsilon>0\) so that all differentials with
		\(\lVert z\rVert<2\varepsilon\) are uniformly invertible. Since the
		second derivatives of \(\Theta_\lambda\) are uniformly bounded on this
		set, the quantitative inverse function theorem gives \(\delta>0\) such
		that, after decreasing \(\varepsilon\) if necessary,
		\(\Theta_\lambda\) is injective on every set
		\[
		(q-\delta,q+\delta)\times B_N(0,\varepsilon).
		\]
		
		To compare pieces with separated parameters, set
		\[
		d_\delta
		:=\lambda\inf_{|s|\geq\delta}
		\lVert h-h(\,\cdot-s)\rVert_{L^2}>0.
		\]
		Indeed, distinct translates of a nonzero compactly supported function are
		distinct, while translates with sufficiently separated supports are
		orthogonal. Decrease \(\varepsilon\) once more so that
		\(2\varepsilon<d_\delta\). If
		\(\Theta_\lambda(q,z)=\Theta_\lambda(q',z')\) and
		\(|q-q'|\geq\delta\), then
		\[
		d_\delta
		\leq\lambda\lVert k(q)-k(q')\rVert
		=\lVert U(q)z-U(q')z'\rVert
		<2\varepsilon,
		\]
		a contradiction. Thus \(\Theta_\lambda\) is globally injective. Since it
		is a local diffeomorphism, it is a diffeomorphism onto its open image.
		The formula
		\[
		d\Theta_\lambda(q,z)(a,\zeta)
		=a\bigl(\lambda k'(q)+U'(q)z\bigr)+U(q)\zeta
		\]
		and the preceding uniform invertibility estimate give the asserted bounds
		on \(d\Theta_\lambda\) and its inverse.
		
		\textbf{\Cref{it:bump-tube-separation}.}
		If \(|q-q'|\geq\delta\), then
		\[
		\bigl\lVert
		\Theta_\lambda(q,z)-\Theta_\lambda(q',z')
		\bigr\rVert
		\geq d_\delta-2\varepsilon>0.
		\]
	Every sequence \(q_n\to\infty\) has a \(\delta\)-separated subsequence,
	so the corresponding image points are uniformly separated. If
	\((\Theta_\lambda(q_n,z_n))\) had a convergent subsequence, then the
	associated subsequence of \((q_n)\) would still tend to infinity and would
	therefore admit a \(\delta\)-separated subsubsequence. Its image would be
	uniformly separated, contradicting convergence. Hence
	\((\Theta_\lambda(q_n,z_n))\) has no convergent subsequence.
	\end{proof}
	
\subsubsection{Proof of \texorpdfstring{\Cref{prop:localized-atkin-end}}
		{the localized-end proposition}}
\label{app:localized-atkin-end-proof}
	
	\begin{proof}[Proof of \Cref{prop:localized-atkin-end}]
		Fix the data
		\[
		\cH_{\mathrm A},\quad
		f_{\mathrm A},\quad
		S,\quad
		\alpha,\quad
		a_{\mathrm{rad}},\quad
		\delta_{\mathrm A},\quad
		c_{\mathrm A}
		\]
		from \Cref{lem:atkin-radial-data}, put \(\cK:=L^2(\bR)\), and identify
		\(\cH\) isometrically with \(\cH_{\mathrm A}\oplus\cK\).
		
		We first embed \(\cH_{\mathrm A}\) into a bounded set while recording its
		radial variable in a translated bump. Choose
		\(0\ne h_0\in C^\infty_c(\bR)\), set
		\(k(q):=h_0(\,\cdot-q)\), and write
		\[
		q(x):=\sqrt{1+\lVert x\rVert_{\cH_{\mathrm A}}^2}.
		\]
		Choose \(\lambda>0\) so that
		\[
		\lambda\sqrt{1+\lVert h_0\rVert_{L^2}^2}<\frac r4,
		\]
		and define
		\[
		E(x):=\lambda\left(\frac{x}{q(x)},k(q(x))\right).
		\]
		The map \(x\mapsto x/q(x)\) is a diffeomorphism onto the open unit ball,
		and both it and \(q\) have bounded differential. Since \(k'\) has constant
		norm, \(E\) is a smooth embedding with uniformly bounded differential and
		image in \(B_{\cH_{\mathrm A}\oplus\cK}(0,r/4)\).
		
		Fix \(q_0>1\), apply \Cref{lem:translated-bump-tube} with \(h=h_0\) and
		\(q_*=q_0\), and write
		\[
		\Theta_\lambda(q,z)=\lambda k(q)+U(q)z.
		\]
		Put
		\[
		r_\lambda(q):=\lambda\sqrt{1-q^{-2}},
		\]
		and choose \(a_0>0\) and \(\varepsilon_0>0\), with
		\(\varepsilon_0\) smaller than the tube radius, such that
		\[
		a_0<\frac12r_\lambda(q_0),
		\qquad
		a_0^2+\varepsilon_0^2<\frac{r^2}{16}.
		\]
		Define
		\[
		\Xi(q,e,a,z)
		:=\bigl((r_\lambda(q)+a)e,\Theta_\lambda(q,z)\bigr)
		\]
		on
		\[
		(q_0,\infty)\times\mathbb S(\cH_{\mathrm A})
		\times(-a_0,a_0)\times B_N(0,\varepsilon_0).
		\]
		The map \(\Xi\) is a diffeomorphism onto an open subset of
		\(B_{\cH_{\mathrm A}\oplus\cK}(0,r)\), and its inverse differential is
		uniformly bounded. Indeed, its central part has norm less than \(r/4\),
		whereas the perturbation \((ae,U(q)z)\) has norm less than \(r/4\).
		Thus its image is contained in the stated ball. Moreover, the
		\(\cK\)-component determines \((q,z)\) with
		uniformly bounded inverse differential by
		\ref{it:bump-tube-coordinates} of
		\Cref{lem:translated-bump-tube}. Since \(r_\lambda(q)+a\) is uniformly
		bounded away from zero, the first component then determines
		\[
		e=\frac{u}{\lVert u\rVert},
		\qquad
		a=\lVert u\rVert-r_\lambda(q).
		\]
		
		Let \(\cV\) be the image of \(\Xi\), set \(Q=q\), and put
		\[
		\mathcal R:=a^2+\lVert z\rVert_{\cK}^2.
		\]
		Choose \(q_1>q_0\) and
		\[
		0<\varepsilon_1<\varepsilon_2<\min\{a_0,\varepsilon_0\}.
		\]
		Let \(\chi\in C^\infty(\bR,[0,1])\) vanish on a neighborhood of
		\((-\infty,q_0]\) and satisfy \(\chi=1\) on \([q_1,\infty)\). Choose
		\(\zeta\in C^\infty([0,\infty),[0,1])\) such that
		\[
		\zeta=1\quad\text{on }[0,\varepsilon_1^2],
		\qquad
		\zeta=0\quad\text{on }[\varepsilon_2^2,\infty),
		\]
		and define
		\[
		\vartheta:=\chi(Q)\zeta(\mathcal R).
		\]
		In the collar coordinates, define
		\[
		X(q,e,a,z):=\sqrt{q^2-1}\,e,
		\qquad
		F:=f_{\mathrm A}\circ X,
		\]
		and let
		\[
		Y(u,w):=(Su,0).
		\]
		Finally, set
		\[
		\delta:=\delta_{\mathrm A},
		\qquad
		h(q):=
		\frac{c_{\mathrm A}}
		{\Psi\bigl(a_{\mathrm{rad}}(q^2-1)\bigr)}
		\quad(q\geq q_0),
		\]
		and choose
		\[
		\Lambda_0
		\geq
		\max\left\{
		\sup_{\cV}F,\,
		\delta^{-1}\sup_{q\geq q_0}h(q)
		\right\}.
		\]
		
		We first verify the bounds and extension properties. The uniform bound on
		\(d\Xi^{-1}\) gives uniform bounds on \(dQ\) and \(d\mathcal R\), and hence
		on \(d\vartheta\). The functions \(F\) and \(h\) are bounded by
		\Cref{lem:atkin-radial-data} and the definition of \(h\), respectively.
		Moreover,
		\[
		\sup_{\cV}\lVert Y\rVert\leq r\lVert S\rVert,
		\qquad
		\sup_{\cV}\lVert dY\rVert\leq\lVert S\rVert.
		\]
		The cutoff \(\vartheta\) vanishes near the boundary faces \(q=q_0\),
		\(|a|=a_0\), and \(\lVert z\rVert=\varepsilon_0\). At the remaining end,
		where \(Q\to\infty\), the \(\cK\)-component has no ambient accumulation
		point by \ref{it:bump-tube-separation} of
		\Cref{lem:translated-bump-tube}. Thus \(\vartheta\) extends smoothly by
		zero to \(\cH\).
		
		For \(\Lambda\geq\Lambda_0\), define on \(\cV\)
		\[
		\rho_\Lambda:=\vartheta F+(1-\vartheta)\Lambda.
		\]
		Then
		\[
		\rho_\Lambda-\Lambda=\vartheta(F-\Lambda).
		\]
		The right-hand side vanishes near every finite boundary face, while the end
		\(Q\to\infty\) has no ambient accumulation point. Extending it smoothly by
		zero, or equivalently setting \(\rho_\Lambda=\Lambda\) outside \(\cV\),
		gives a smooth function on \(\cH\). Positivity of \(F\) and
		\(0\leq\vartheta\leq1\) also give
		\[
		0<\rho_\Lambda\leq\Lambda,
		\qquad
		\rho_\Lambda\geq\Lambda(1-\vartheta).
		\]
		
		We next establish the virial estimate. The flow of \(Y\) is
		\[
		(u,w)\longmapsto(e^{tS}u,w).
		\]
		It preserves \(Q\) and \(\mathcal R\), and hence \(\vartheta\), while
		\[
		X(e^{tS}u,w)=e^{tS}X(u,w).
		\]
		Consequently,
		\[
		d\vartheta(Y)=0,
		\qquad
		dF(Y)=d(f_{\mathrm A})_X(SX).
		\]
		Since
		\[
		\lVert X\rVert_{\cH_{\mathrm A}}^2=Q^2-1,
		\]
		\eqref{eq:atkin-virial} gives
		\[
		\frac12\,dF(Y)+\delta F
		\geq
		\frac{c_{\mathrm A}}
		{\Psi\bigl(a_{\mathrm{rad}}(Q^2-1)\bigr)}
		=h(Q).
		\]
		For all sufficiently large \(q\), there is a constant \(c>0\) such that
		\[
		\sqrt{h(q)}\geq\frac{c}{q\log q}.
		\]
		It follows that
		\[
		\int_{q_0}^{\infty}\sqrt{h(q)}\,dq=\infty.
		\]
		If \(Q\leq R\), then
		\(\lVert X\rVert_{\cH_{\mathrm A}}^2\leq R^2-1\), and hence
		\[
		F
		\geq
		\frac{
			\exp\bigl(-\lVert C\rVert_{\mathrm{op}}(R^2-1)\bigr)
		}{
			\Psi\bigl(a_{\mathrm{rad}}(R^2-1)\bigr)
		}>0.
		\]
		Denote the positive constant on the right by \(\widehat m_R\).
		Since \(d\vartheta(Y)=0\), the preceding estimate and the choice
		of \(\Lambda_0\) imply
		\[
		\begin{aligned}
			\frac12\,d\rho_\Lambda(Y)+\delta\rho_\Lambda
			&=
			\vartheta\left(\frac12\,dF(Y)+\delta F\right)
			+(1-\vartheta)\delta\Lambda\\
			&\geq h(Q).
		\end{aligned}
		\]
		If in addition \(\vartheta\geq1/4\), the displayed lower bound for \(F\)
		gives
		\[
		\rho_\Lambda\geq\vartheta F\geq\frac14\widehat m_R
		\qquad\text{on }\{Q\leq R\}.
		\]
		This proves \ref{it:localized-end-factor}.
		
		It remains to construct the curve in
		\ref{it:localized-end-escape}. For \(x\ne0\) with \(q(x)>q_0\), the point
		\(E(x)\) has collar coordinates
		\[
		\bigl(q(x),x/\lVert x\rVert,0,0\bigr),
		\]
		so that \(X(E(x))=x\). Since
		\(\lVert\alpha(t)\rVert_{\cH_{\mathrm A}}\to\infty\), one has
		\[
		Q(E(\alpha(t)))=q(\alpha(t))\longrightarrow\infty.
		\]
		Choose \(t_*>0\) so that \(\vartheta(E(\alpha(t)))=1\) for
		\(t\geq t_*\), and define
		\(\beta(t):=E(\alpha(t+t_*))\). Then \(\vartheta\circ\beta=1\), so
		\(\rho_\Lambda\circ\beta=F\circ\beta\) for every
		\(\Lambda\geq\Lambda_0\), and
		\[
		L_{\rho_\Lambda\langle\cdot,\cdot\rangle_{\cH}}(\beta)
		\leq
		\sup_x\lVert dE_x\rVert\,
		L_{G_{\mathrm A}}(\alpha)<\infty.
		\]
		For every \(t_n\to\infty\), the \(\cK\)-component of \(\beta(t_n)\) is
		\(\lambda k(q(\alpha(t_n+t_*)))\), where
		\(q(\alpha(t_n+t_*))\to\infty\). By
		\ref{it:bump-tube-separation} of
		\Cref{lem:translated-bump-tube}, this sequence has no convergent
		subsequence. The same is therefore true of \((\beta(t_n))_n\).
	\end{proof}
	
	\bibliographystyle{abbrv}
	\bibliography{ref_new}

\end{document}